\documentclass[a4paper,twoside]{article}
\usepackage{a4}
\usepackage{amssymb}
\usepackage{amsmath}
\usepackage{upref}
\usepackage[active]{srcltx}
\usepackage[pagebackref,colorlinks,citecolor=blue,linkcolor=blue]{hyperref}
\allowdisplaybreaks[2] 
\newcount\minutes \newcount\hours
\hours=\time
\divide\hours 60
\minutes=\hours
\multiply\minutes -60
\advance\minutes \time
\newcommand{\klockan}{\the\hours:{\ifnum\minutes<10 0\fi}\the\minutes}
\newcommand{\tid}{\today\ \klockan}
\newcommand{\prtid}{\smash{\raise 10mm \hbox{\LaTeX ed \tid}}}
\renewcommand{\prtid}{}
\makeatletter
\def\sectionmark#1{} 
\def\subsectionmark#1{}
\newcommand{\sectnr}{\ifnum \c@secnumdepth >\z@
                 \thesection.\hskip 1em\relax \fi}
\def\@evenhead{\footnotesize\rm\thepage\hfil\leftmark\hfil\llap{\prtid}}
\def\@oddhead{\footnotesize\rm\rlap{\prtid}\hfil\rightmark\hfil\thepage}
\def\tableofcontents{\section*{Contents} 
 \@starttoc{toc}}
\makeatother
\makeatletter
\def\@biblabel#1{#1.}
\makeatother
\makeatletter
\let\Thebibliography=\thebibliography
\renewcommand{\thebibliography}[1]{\def\@mkboth##1##2{}\Thebibliography{#1}
\addcontentsline{toc}{section}{References}
\frenchspacing 
\setlength{\@topsep}{0pt}
\setlength{\itemsep}{0pt}%
\setlength{\parskip}{0pt plus 2pt}%
}
\makeatother
\makeatletter
\def\mdots@{\mathinner.\nonscript\!.%
 \ifx\next,.\else\ifx\next;.\else\ifx\next..\else
 \nonscript\!\mathinner.\fi\fi\fi}
\let\ldots\mdots@
\let\cdots\mdots@
\let\dotso\mdots@
\let\dotsb\mdots@
\let\dotsm\mdots@
\let\dotsc\mdots@
\def\vdots{\vbox{\baselineskip2.8\p@ \lineskiplimit\z@
    \kern6\p@\hbox{.}\hbox{.}\hbox{.}\kern3\p@}}
\def\ddots{\mathinner{\mkern1mu\raise8.6\p@\vbox{\kern7\p@\hbox{.}}%
    \raise5.8\p@\hbox{.}\raise3\p@\hbox{.}\mkern1mu}}
\makeatother
\makeatletter
\let\Enumerate=\enumerate
\renewcommand{\enumerate}{\Enumerate%
\setlength{\@topsep}{0pt}
\setlength{\itemsep}{0pt}%
\setlength{\parskip}{0pt plus 1pt}%
\renewcommand{\theenumi}{\textup{(\alph{enumi})}}%
\renewcommand{\labelenumi}{\theenumi}%
}
\let\endEnumerate=\endenumerate
\renewcommand{\endenumerate}{\endEnumerate\unskip}

\makeatother
\makeatletter
\def\@seccntformat#1{\csname the#1\endcsname.\quad}
\makeatother
\makeatletter
\long\def\@makecaption#1#2{%
  \vskip\abovecaptionskip
  \sbox\@tempboxa{ #1. #2}%
  \ifdim \wd\@tempboxa >\hsize
    #1. #2\par
  \else
    \global \@minipagefalse
    \hb@xt@\hsize{\hfil\box\@tempboxa\hfil}%
  \fi
  \vskip\belowcaptionskip}
\makeatother
\newcommand{\authortitle}[2]{\author{#1}\title{#2}\markboth{#1}{#2}}
\newcommand{\art}[6]{{\sc #1, \rm #2, \it #3\/ \bf #4 \rm (#5), \mbox{#6}.}}
\newcommand{\auth}[2]{{#2. #1}}
\newcommand{\artprep}[3]{{\sc #1, \rm #2, #3.}}

\newcommand{\book}[3]{{\sc #1, \it #2, \rm #3.}}
\newcommand{\AND}{{\rm and }}
\RequirePackage{amsthm}
\newtheoremstyle{descriptive}%
  {\topsep}   
  {\topsep}   
  {\rmfamily} 
  {}          
  {\bfseries} 
  {.}         
  { }         
  {}          
\newtheoremstyle{propositional}%
  {\topsep}   
  {\topsep}   
  {\itshape}  
  {}          
  {\bfseries} 
  {.}         
  { }         
  {}          
\theoremstyle{propositional}
\newtheorem{thm}{Theorem}[section]
\newtheorem{theo}[thm]{Theorem}   
\newtheorem{prop}[thm]{Proposition}
\newtheorem{lem}[thm]{Lemma}
\newtheorem{cor}[thm]{Corollary}
\theoremstyle{descriptive}
\newtheorem{definition}[thm]{Definition}
\newtheorem{remark}[thm]{Remark}
\makeatletter
\renewenvironment{proof}[1][\proofname]{\par
  \pushQED{\qed}%
  \normalfont
  \trivlist
  \item[\hskip\labelsep
        \itshape
    #1\@addpunct{.}]\ignorespaces
}{%
  \popQED\endtrivlist\@endpefalse
}
\makeatother
{\catcode`p =12 \catcode`t =12 \gdef\eeaa#1pt{#1}}      
\def\accentadjtext#1{\setbox0\hbox{$#1$}\kern   
                \expandafter\eeaa\the\fontdimen1\textfont1 \ht0 }
\def\accentadjscript#1{\setbox0\hbox{$#1$}\kern 
                \expandafter\eeaa\the\fontdimen1\scriptfont1 \ht0 }
\def\accentadjscriptscript#1{\setbox0\hbox{$#1$}\kern   
                \expandafter\eeaa\the\fontdimen1\scriptscriptfont1 \ht0 }
\def\accentadjtextback#1{\setbox0\hbox{$#1$}\kern       
                -\expandafter\eeaa\the\fontdimen1\textfont1 \ht0 }
\def\accentadjscriptback#1{\setbox0\hbox{$#1$}\kern     
                -\expandafter\eeaa\the\fontdimen1\scriptfont1 \ht0 }
\def\accentadjscriptscriptback#1{\setbox0\hbox{$#1$}\kern 
                -\expandafter\eeaa\the\fontdimen1\scriptscriptfont1 \ht0 }
\def\itoverline#1{{\mathsurround0pt\mathchoice
        {\rlap{$\accentadjtext{\displaystyle #1}
                \accentadjtext{\vrule height1.593pt}
                \overline{\phantom{\displaystyle #1}
                \accentadjtextback{\displaystyle #1}}$}{#1}}
        {\rlap{$\accentadjtext{\textstyle #1}
                \accentadjtext{\vrule height1.593pt}
                \overline{\phantom{\textstyle #1}
                \accentadjtextback{\textstyle #1}}$}{#1}}
        {\rlap{$\accentadjscript{\scriptstyle #1}
                \accentadjscript{\vrule height1.593pt}
                \overline{\phantom{\scriptstyle #1}
                \accentadjscriptback{\scriptstyle #1}}$}{#1}}
        {\rlap{$\accentadjscriptscript{\scriptscriptstyle #1}
                \accentadjscriptscript{\vrule height1.593pt}
                \overline{\phantom{\scriptscriptstyle #1}
                \accentadjscriptscriptback{\scriptscriptstyle #1}}$}{#1}}}}
\def\itunderline#1{{\mathsurround0pt\mathchoice
        {\rlap{$\underline{\phantom{\displaystyle #1}
                \accentadjtextback{\displaystyle #1}}$}{#1}}
        {\rlap{$\underline{\phantom{\textstyle #1}
                \accentadjtextback{\textstyle #1}}$}{#1}}
        {\rlap{$\underline{\phantom{\scriptstyle #1}
                \accentadjscriptback{\scriptstyle #1}}$}{#1}}
        {\rlap{$\underline{\phantom{\scriptscriptstyle #1}
                \accentadjscriptscriptback{\scriptscriptstyle #1}}$}{#1}}}}
\newcommand{\limplus}{{\mathchoice{\vcenter{\hbox{$\scriptstyle +$}}}
  {\vcenter{\hbox{$\scriptstyle +$}}}
  {\vcenter{\hbox{$\scriptscriptstyle +$}}}
  {\vcenter{\hbox{$\scriptscriptstyle +$}}}
}}
\newcommand{\limminus}{{\mathchoice{\vcenter{\hbox{$\scriptstyle -$}}}
  {\vcenter{\hbox{$\scriptstyle -$}}}
  {\vcenter{\hbox{$\scriptscriptstyle -$}}}
  {\vcenter{\hbox{$\scriptscriptstyle -$}}}
}}
\newcommand{\limpm}{{\mathchoice{\vcenter{\hbox{$\scriptstyle \pm$}}}
  {\vcenter{\hbox{$\scriptstyle \pm$}}}
  {\vcenter{\hbox{$\scriptscriptstyle \pm$}}}
  {\vcenter{\hbox{$\scriptscriptstyle \pm$}}}
}}
\newdimen\extrawidth
\def\iintlim#1#2{\setbox0\hbox{$\scriptstyle#1$}%
	\setbox1\hbox{$\scriptstyle#2$}%
	\extrawidth=\wd1 \advance\extrawidth-\wd0
	\ifdim\extrawidth<0pt \extrawidth=0pt\fi%
	\int_{#1\kern\extrawidth \kern .5em}^{#2\kern -\wd1} \kern -.5em%
}
\numberwithin{equation}{section}
\newcommand{\imp}{\ensuremath{\mathchoice{\quad \Longrightarrow \quad}{\Rightarrow}
                {\Rightarrow}{\Rightarrow}} }
\newenvironment{ack}{\medskip{\it Acknowledgement.}}{}

\DeclareMathOperator{\diam}{diam}
\DeclareMathOperator{\Div}{div}
\DeclareMathOperator{\dvg}{div}
\newcommand{\capacity}{\operatornamewithlimits{cap}}

\renewcommand{\phi}{\varphi}
\newcommand{\eps}{\varepsilon}
\newcommand{\al}{\alpha}
\newcommand{\be}{\beta}
\newcommand{\de}{\delta}

\newcommand{\R}{\mathbf{R}}
\newcommand{\Q}{\mathbf{Q}}
\newcommand{\Qp}{{\Q_\limplus}}
\newcommand{\Rn}{\mathbf{R}^n}
\newcommand{\Rno}{\mathbf{R}^{n+1}}

\newcommand{\dist}{\operatorname{dist}}

\newcommand{\esssup}{\operatornamewithlimits{ess\, sup}}
\DeclareMathOperator*{\essliminf}{ess\,lim\,inf}
\DeclareMathOperator*{\esslimsup}{ess\,lim\,sup}

\newcommand{\Th}{\Theta}

\newcommand{\p}{{$p\mspace{1mu}$}}

\newcommand{\uS}{\itoverline{S}}
\newcommand{\lS}{\itunderline{S}}
\newcommand{\uP}{\itoverline{P}}
\newcommand{\lP}{\itunderline{P}}
\newcommand{\UU}{\mathcal{U}}%
\newcommand{\UUt}{\widetilde{\mathcal{U}}}%
\newcommand{\LL}{\mathcal{L}}%
\newcommand{\LLt}{\widetilde{\mathcal{L}}}%
\newcommand{\bdy}{\partial}
\newcommand{\bdry}{\partial}
\newcommand{\bdyp}{\bdy_p}
\newcommand{\grad}{\nabla}
\newcommand{\setm}{\setminus}
\renewcommand{\emptyset}{\varnothing}
\newcommand{\alp}{\alpha}
\newcommand{\ga}{\gamma}
\newcommand{\gat}{\widetilde{\gamma}}
\newcommand{\tf}{\tilde{f}}
\newcommand{\wcj}{w_{c,j}}
\newcommand{\wck}{w_{c,k}}
\newcommand{\ucj}{u_{c,j}}
\newcommand{\vcj}{v_{c,j}}
\newcommand{\psicj}{\psi_{c,j}}

\newcommand{\gh}{\hat{g}}
\newcommand{\gb}{\bar{g}}
\newcommand{\tb}{\tau}
\newcommand{\hh}{\hat{h}}
\newcommand{\hb}{\bar{h}}
\newcommand{\uh}{\hat{u}}
\newcommand{\uj}{u_j}
\newcommand{\la}{\lambda}
\newcommand{\ut}{\tilde{u}}
\newcommand{\ft}{\tilde{f}}
\newcommand{\gah}{\widehat{\gamma}}
\newcommand{\Qjt}{Q_j^\tau}

\begin{document}
%
%
\authortitle{Anders Bj\"orn, Jana Bj\"orn, Ugo Gianazza
and Juhana Siljander}
{Boundary regularity for the porous medium equation}
\author{
Anders Bj\"orn \\
\it\small Department of Mathematics, Link\"oping University, \\
\it\small SE-581 83 Link\"oping, Sweden\/{\rm ;}
\it \small anders.bjorn@liu.se
\\
\\
Jana Bj\"orn \\
\it\small Department of Mathematics, Link\"oping University, \\
\it\small SE-581 83 Link\"oping, Sweden\/{\rm ;}
\it \small jana.bjorn@liu.se
\\
\\
Ugo Gianazza \\
\it\small Department of Mathematics ``F. Casorati'', Universit\`a di Pavia,\\
\it\small via Ferrata 1, IT-27100 Pavia, Italy\/{\rm ;}
\it\small gianazza@imati.cnr.it
\\
\\
Juhana Siljander \\
\it\small Department of Mathematics and Statistics, University of Jyv\"askyl\"a,\\
\it\small P.O. Box 35, FI-40014 Jyv\"askyl\"a, Finland\/{\rm ;}
\it\small juhana.siljander@jyu.fi 
}

\date{}
\maketitle

\noindent{\small
{\bf Abstract}. 
We study the boundary regularity of solutions to the porous medium
equation $u_t = \Delta u^m$ in the degenerate range $m>1$. 
In particular, we show that in cylinders 
the Dirichlet problem
with positive continuous boundary data on the parabolic boundary
has a solution which attains the boundary values, 
provided that the spatial domain satisfies the elliptic Wiener
criterion. This condition is known to be optimal,
and it is a consequence of our main theorem which
establishes a barrier characterization of regular boundary points for
general -- not necessarily cylindrical -- domains in $\R^{n+1}$. 
One of our fundamental tools is a new strict comparison principle
between sub- and superparabolic functions,
which makes it essential for us to study
both nonstrict and strict Perron solutions to be able
to develop a fruitful boundary regularity theory.
Several other comparison principles and pasting lemmas are also obtained.
In the process we obtain a rather complete picture of the relation
between sub/super\-para\-bolic functions and weak sub/super\-solu\-tions.
}



\section{Introduction}

Let $\Theta$ be a bounded open set in a Euclidean space
and for every $f \in C(\bdy \Theta)$
let $u_f$ be the solution of the Dirichlet problem 
with boundary data $f$
for a given partial differential equation.
Then a boundary point $\xi_0 \in \bdy \Theta$ is \emph{regular} if
\[
             \lim_{\Theta \ni \zeta \to \xi_0} u_f(\zeta)=f(\xi_0)
             \quad \text{for all } f \in C(\bdy \Theta),
\]
i.e.\ if the solution to the Dirichlet problem attains
the given boundary data continuously at $\xi_0$,
for all continuous boundary data $f$.

In this paper, we characterize 
regular boundary points for the porous medium equation
\begin{equation}   \label{eq-porous-intro}
\bdy_t u = \Delta u^m
\end{equation}
in terms of families of barriers, in the so-called \emph{degenerate case} 
$1<m<\infty$,
and for general (not necessarily cylindrical) domains.
To our knowledge, 
Abdulla~\cite{abdulla-00}, \cite{abdulla-05} is the only one who
has studied 
the Dirichlet problem for the porous medium equation 
in noncylindrical domains.

The characterization of regular boundary points for different 
partial differential equations has a very long history.
Poincar\'e~\cite{poincare1890} was the first to use barriers, while
Lebesgue~\cite{lebesgue1912a} coined the name.
At that time, barriers
were used to study the solvability of the Dirichlet problem
for harmonic functions,
a question that was later completely settled using e.g.\ Perron solutions.
In 1924, Lebesgue~\cite{lebesgue1924} characterized regular boundary points
for harmonic functions by the existence of barriers.
The corresponding  characterization for the heat equation 
was given by Bauer~\cite{Bauer62} in 1962, but barriers
had then already been used 
to study boundary regularity  for the heat equation
since Petrovski\u\i~\cite{Petro2} 
in  1935; 
see the introduction in \cite{BBP} for more on the
history of boundary regularity for the heat equation.

Coming to nonlinear parabolic equations of degenerate and singular
types, 
the potential theory for \p-parabolic 
equations was
initiated by Kilpel\"ainen and Lindqvist in \cite{KiLi96}. They
established the parabolic Perron method, and also suggested a boundary
regularity characterization in terms of one barrier. Even if the
single barrier criterion has turned out to be  
problematic, \cite{KiLi96} has been the basis for the
further development by Lindqvist~\cite{lindqvist95},
Bj\"orn--Bj\"orn--Gianazza~\cite{BBG}
and
Bj\"orn--Bj\"orn--Gianazza--Parviainen~\cite{BBGP}
for the \p-\emph{parabolic equation} 
\begin{equation} \label{eq-p-para}
   \partial_t u = \Delta_p u := \Div(|\nabla u|^{p-2} \nabla u).
\end{equation}

For the porous medium equation~\eqref{eq-porous-intro}, 
potential theory is largely at its
inception, and so far not very much is known about the boundary
behaviour of solutions in general domains.
To our knowledge the main
contributions in this field are
due to Ziemer~\cite{ziemer}, Abdulla~\cite{abdulla-00}, \cite{abdulla-05}
and Kinnunen--Lindqvist--Lukkari~\cite{Kinnunen-Lindqvist-Lukkari:2013}.

Ziemer~\cite{ziemer} studied boundary regularity in cylinders
for a class of  degenerate parabolic equations,
which includes the porous medium equation 
with $m>1$,
but with boundary data taken in a weak (Sobolev) sense;
see Section~\ref{S:reg-cylinders} for further details.

 Abdulla~\cite{abdulla-00}, \cite{abdulla-05}
investigated the Dirichlet problem for the
porous medium equation 
with $m>0$ in general 
domains $\Theta\subset\R^{n+1}$, $n\geq 2$. 
Existence was established in \cite{abdulla-00},
while  uniqueness, comparison and stability theorems were
presented in \cite{abdulla-05}. 
Therein, the smoothness condition on the boundary in order to 
have $u\in C(\overline{\Theta})$
is given in terms of a parabolic H\"older-type modulus; 
cf.\ Theorems~\ref{thm-cont-exist} and~\ref{thm-stability}
for the cylindrical case.

Kinnunen--Lindqvist--Lukkari~\cite{Kinnunen-Lindqvist-Lukkari:2013}
developed the Perron method for the porous medium equation in the 
degenerate range $m>1$ and showed that nonnegative
continuous boundary functions are resolutive in arbitrary cylindrical
domains. 
A boundary function $f$ is \emph{resolutive} if the upper and lower Perron
solutions $\uP f$ and $\lP f$ coincide.

The present paper can be considered as an extension of the previous
contributions in several different but strictly related directions, 
as well as an 
initial development of a boundary 
regularity theory for the
porous medium equation in terms of barriers. Under this
second point of view, it is strictly related to the works
\cite{BBG} and~\cite{BBGP} for the \p-parabolic equation
\eqref{eq-p-para}, 
even though the porous medium equation has extra
difficulties not present for the \p-parabolic equation.
In particular, if $u$ is a solution of the porous medium 
equation~\eqref{eq-porous-intro}
and $c \ne 0$ is a constant,
then typically $u+c$ is not a solution. 
Moreover, we restrict ourselves to
nonnegative functions, and therefore are not allowed to change
sign. 

It is possible to
study sign-changing solutions of the porous medium
equation, as has been done by some authors, but in addition
to causing extra difficulties it may also cause significant differences 
when it comes to 
boundary regularity,
as it seems quite possible that boundary regularity can be different for
nonnegative and sign-changing functions.
Here we restrict ourselves to nonnegative, and primarily positive, functions.

A well-known problem for the 
porous medium equation is the difficulty of obtaining a comparison principle
between sub- and superparabolic functions.
One of the main achievements in \cite{Kinnunen-Lindqvist-Lukkari:2013} 
was their comparison principle for cylinders
(cf.\ Theorem~\ref{thm-comp-KLL}).
In order to even start developing the theory in this paper, it
is fundamental to have a comparison principle in general domains,
which we obtain in Theorem~\ref{thm-para-comp-princ}.

Comparison principles usually require an inequality $\le$ on the boundary, 
and to establish such a comparison principle for general domains has been
a major problem both for earlier authors and for us. 
We have chosen a slightly different and novel route obtaining a strict
comparison principle in general domains, 
with the strict inequality $<$ at the boundary
(see Theorems~\ref{thm-para-comp-princ} and~\ref{thm-comp-princ}). 
Using a strict comparison principle
 causes extra complications, but 
we have still been able to develop a fruitful Perron and boundary regularity
theory in general domains.

For thorough presentations of the theory of the porous
medium equation, we refer the interested reader to
Daskalopoulos--Kenig~\cite{Daskalopoulos-Kenig-book} and 
V\'azquez~\cite{Vazquez-book-1};
see also 
DiBenedetto--Gianazza--Vespri~\cite{DBGV-mono}.
We primarily deal with the degenerate case $m\ge1$, but whenever possible we have
given statements for general $m>0$. The \emph{singular case} $0<m<1$
will be the object of future research. 

The paper is  organized as follows.
 Section~\ref{S:Prelim}
is devoted to some preliminary material. In particular, we recall the
different concepts of solutions and sub/super\-solu\-tions,
as well as various existence, uniqueness and stability results that will
be essential later on.

Section~\ref{S:Superparabolic} deals with the notions of sub- and
superparabolic functions. 
In Theorem~\ref{thm-lsc-supersoln}, we
show that if $u$ is a
weak supersolution then its \emph{lsc-regularization} $u_*$ is
superparabolic. 
A corresponding result for weak subsolutions is also obtained.
(As we are not allowed to change sign, the theory for weak subsolutions
does not follow directly from the corresponding theory for weak
supersolutions.)
We conclude the section by presenting the parabolic comparison
principle for cylinders due to 
Kinnunen--Lindqvist--Lukkari~\cite{Kinnunen-Lindqvist-Lukkari:2013},
with a new proof. 

In Section~\ref{S:Further-Superparabolic} we consider further results
on sub/super\-para\-bolic functions: in particular, under proper
conditions, sub/super\-para\-bolic functions are weak
sub/super\-solu\-tions. In this way, 
we establish a rather complete understanding of
the relation between weak sub/super\-solu\-tions
and sub/super\-para\-bolic functions.

Section~\ref{S:Comparison} is devoted to a series of different
comparison principles, for sub- and superparabolic functions, both of
elliptic and parabolic types, and both of strict and nonstrict 
types. 
Several pasting lemmas are also obtained.

In Section~\ref{S:Boundary} we deal with the Perron method, and with
boundary regularity. We introduce the notion of upper regular points,
as well as of lower regular points for positive (resp.\ nonnegative) boundary
data. From here on we restrict ourselves to bounded open  sets
$\Theta\subset\R^{n+1}$. 
Moreover, the boundary data are always
assumed to be bounded. 

Section~\ref{S:Upper-reg} is devoted to the characterization of an
upper regular point in terms of a two-parameter family of barriers,
with some related properties, whereas Section~\ref{S:Lower-reg} deals
with the characterization of a lower regular point for positive
boundary data, in terms of another  two-parameter family of
barriers. 
This reflects the fact that we can neither add
constants nor change sign, which is the
crucial difference compared with the \p-parabolic
equation \eqref{eq-p-para}, where 
a single one-parameter family of barriers is necessary and sufficient
(see \cite{BBG} and~\cite{BBGP}).  
In this paper, we do not develop the general theory
of lower regularity for nonnegative boundary data. 
 
In Section~\ref{S:Earlier} we show that the earliest points are always
regular, while  in Section~\ref{S:Future} we
prove that upper regularity, as well as lower regularity (for positive
boundary data), are independent of the future.  

Section~\ref{S:reg-cylinders} collects the most important
contributions of the paper. 
First, we show in
Theorem~\ref{thm-reg-cylinder} that the boundary regularity 
(for positive boundary data)
of a lateral boundary point $(x_0,t_0) \in \bdy U \times [t_1,t_2]$,
with respect to the cylinder $U \times (t_1,t_2)$, is determined by the
elliptic regularity of  $x_0$ with respect to the spatial set $U$. 
This result is
optimal in the sense that every harmonic function $u$ induces a
time-independent solution $u^{1/m}$ of the porous medium equation, and
the Wiener criterion is a necessary and sufficient condition for
boundary regularity of harmonic functions. 
Then, in
Theorem~\ref{thm-reg-monot-cyl} we give a unique solvability result in
suitable finite unions of cylinders, which generalizes previous unique
solvability results due to Abdulla~\cite{abdulla-00}, \cite{abdulla-05},
as well as the resolutivity result by
Kinnunen--Lindqvist--Lukkari~\cite{Kinnunen-Lindqvist-Lukkari:2013}
for general cylinders. 

Finally, Appendix~\ref{app-pf-usc} is devoted to the proof of
Theorem~\ref{thm-usc-repr}; we thought it better to postpone it, in
order not to spoil the flow of the main arguments in
Section~\ref{S:Superparabolic}. 

\begin{ack}
This research started while the authors
were visiting Institut Mittag-Leffler in 2013,
we thank the institute for the kind hospitality.
\end{ack}

\section{Preliminaries}\label{S:Prelim}

Let $\Theta$ be an open set in $\R^{n+1}$, $n\ge 2$.
We write points in $\R^{n+1}$ as $\xi=(x,t)$, where $x \in \R^n$ and $t \in \R$.
For $m>0$, we consider the \emph{porous medium equation}
\begin{equation} \label{eq:para}
\bdy_t u=\Delta u^m :=\Div(\nabla u^{m}),
\end{equation}
where, from now on, the \emph{gradient} $\grad$ and the 
\emph{divergence} $\Div$
are taken with respect to $x$. 
In this paper we only consider nonnegative solutions $u$.
This equation is \emph{degenerate} if $m>1$ and \emph{singular} if $0<m<1$.
For $m=1$ it is the usual heat equation.
Observe that if $u$ satisfies \eqref{eq:para},
and $a \in \R_\limplus$, then (in general) $au$ and $u+a$ do not
satisfy \eqref{eq:para}.

All our cylinders are bounded space-time cylinders, i.e.\
of the form $U_{t_1,t_2}:=U\times (t_1,t_2) \Subset \R^{n+1}$,
where $U \Subset\R^n$ is open.
We say that $U_{t_1,t_2}$ is a \emph{$C^{k,\al}$-cylinder}
if $U$ is $C^{k,\al}$-smooth.
The \emph{parabolic boundary} of 
$U_{t_1,t_2}$ is 
\[
\partial_p U_{t_1,t_2}=(\overline U\times \{t_1\})\cup(\partial U\times (t_1,t_2]).
\]
We define the \emph{parabolic boundary} of a finite union of open
cylinders $U^j_{t_j,s_j}$ as follows
\[
\partial_p \biggl(\bigcup_{j=1}^N  U^j_{t_j,s_j}\biggr) : =
\biggl(\bigcup_{j=1}^N \partial_p U^j_{t_j,s_j}\biggr) \setminus
\bigcup_{j=1}^N{U}^j_{t_j,s_j}.
\]
Note that the parabolic boundary is by definition compact.
Further, $B(x,r)=\{z \in \R^{n} : |z-x| <r\}$
stands for the usual Euclidean ball in $\R^n$.
We also let
\begin{align*}
\Theta_T &= \{ (x,t)\in\Theta: t<T\}, \\
\Theta_\limminus&=\{(x,t) \in \Theta: t <0\},\\
 \Theta_\limplus&=\{(x,t) \in \Theta: t >0\}.
\end{align*}

Let $U$ be a bounded open set in $\Rn$.  
As usual, $W^{1,2}(U)$
denotes the space of real-valued functions $u$ such that 
$u \in L^2(U)$ and the distributional first partial derivatives 
$\partial u /\partial x_i$,
$i=1,2,\dots,n$, exist in $U$ 
and belong to $L^2(U)$. 
We use the norm
\[
\|u\|_{W^{1,2}(U)} =\biggl(\int_U|u|^2\,dx + \int_U|\nabla u|^2\,dx\biggr)^{1/2}.
\]
The Sobolev space $W_0^{1,2}(U)$ with zero boundary values is
the closure of $C_0^\infty(U)$ with respect to the Sobolev norm.

By the \emph{parabolic Sobolev space} $L^2(t_1,t_2;W^{1,2}(U))$,
with $t_1<t_2$, we mean the space of measurable
functions $u(x,t)$ such that the mapping 
$x\mapsto u(x,t)$ belongs to $W^{1,2}(U)$ for a.e.\ 
$t_1<t<t_2$ and the norm
\[
\biggl(\iintlim{t_1}{t_2}\int_U |u(x,t)|^{2} + |\nabla
u(x,t)|^{2}\,dx\,dt\biggr)^{1/2}
\]
is finite. 
The definition of the space
$L^2(t_1,t_2;W_{0}^{1,2}(U))$ is similar.
Analogously, by the space $C(t_1,t_2;L^2(U))$,
with $t_1<t_2$, we mean the space of measurable
functions $u(x,t)$, such that the mapping
$t\mapsto u(\,\cdot\,,t) \in L^2(U)$ is continuous in the time interval $[t_1,t_2]$.
We can now introduce the notion of weak solution.

\begin{definition}\label{def-sol}
A function $u:\Theta \to [0,\infty]$ is a  
\emph{weak solution} of equation \eqref{eq:para} 
if whenever $U_{t_1,t_2} \Subset \Theta$, 
we have 
$u \in C(t_1,t_2;L^{m+1}(U))$, $u^m\in L^{2}(t_1,t_2;W^{1,2}(U))$ and  
$u$ satisfies the integral equality
\begin{equation} \label{eq:weak-solution}
\iintlim{t_1}{t_2}\int_{U} { \nabla u^{m}} \cdot
\nabla\phi \, dx\,dt - \iintlim{t_1}{t_2}\int_{U} u
\partial_t\phi \, dx\,dt   =  0
\quad \text{for all }\phi \in C_0^\infty(U_{t_1,t_2}).
\end{equation}
Continuous weak solutions are called  \emph{parabolic functions}.

A function $u:\Theta \to [0,\infty]$
is a  \emph{weak supersolution} (\emph{subsolution})
if whenever $U_{t_1,t_2} \Subset
\Theta$, we have $u^m \in L^{2}(t_1,t_2;W^{1,2}(U))$ and the 
left-hand side 
above is nonnegative (nonpositive) for all
nonnegative $\phi \in C_0^\infty(U_{t_1,t_2})$.
\end{definition}

One can also consider sign-changing (and nonpositive) weak
(sub/super)solu\-tions, defined analogously, see 
Kinnunen--Lindqvist~\cite{Kinnunen-Lindqvist:2008} for details.
The general sign-changing theory is however much less developed
than the theory for nonnegative functions.
Moreover, it seems likely that regularity for sign-changing solutions
of the porous medium equation may be quite different from regularity 
when restricted to positive or nonnegative solutions, which we have chosen
to work with here.
For simplicity, we will often
omit \emph{weak}, when talking of weak (sub/super)\-solu\-tions. 

In this paper, the name parabolic (and later sub/super\-para\-bolic)
refers precisely to the porous medium equation \eqref{eq:para}, 
which is just one
of many parabolic equations considered in the literature.
A more specific terminology could be ``porous-parabolic''
but for simplicity and readability
we refrain from this nonstandard term.

\begin{remark}
In Definition~\ref{def-sol}, when dealing with the range $m>1$, 
one could actually require less (see below) on $u$, namely
\begin{equation}\label{alt-def}
u \in C(t_1,t_2;L^2(U))
\quad\text{and}\quad
 u^{(m+1)/2}\in L^{2}(t_1,t_2;W^{1,2}(U)).
\end{equation}
This has been done e.g.\ in DiBenedetto--Gianazza--Vespri~\cite{DBGV-mono}. 
Roughly speaking, our notion of solution corresponds to using 
$u^m$ as a test function in the weak formulation \eqref{eq:weak-solution}, 
whereas assuming \eqref{alt-def} amounts to using $u$. 
Such a choice seems more natural in a number of applications, 
but it seemingly introduces the extra difficulty
that two different notions of solutions are needed, according to 
whether $m \le 1$ or $m \ge 1$.
However, it has recently been proved by
B\"ogelein--Lehtel\"a--Sturm~\cite[Theorem~1.2]{BLS-2018},
that for $m\ge1$ the two notions are equivalent.
\end{remark}

Locally bounded solutions are locally H\"older continuous: this result is due to
different authors. A full account is given in 
Daskalopoulos--Kenig~\cite{Daskalopoulos-Kenig-book}, 
DiBenedetto--Gianazza--Vespri~\cite{DBGV-mono}
and V\'azquez~\cite{Vazquez-book-1}. 
For $m>\frac{(n-2)_\limplus}{n+2}$
solutions are automatically 
locally bounded, 
whereas for $0<m\le\frac{(n-2)_\limplus}{n+2}$ 
explicit unbounded solutions are known, and in order to guarantee  
boundedness, an extra assumption on $u$ is needed 
(see the discussions in DiBenedetto~\cite[Chapter~V]{dibe-mono} 
and DiBenedetto--Gianazza--Vespri~\cite[Appendix~B]{DBGV-mono}). 
Although it plays no role in the following, it is worth mentioning 
that nonnegative solutions satisfy proper forms of Harnack inequalities 
(see \cite{DBGV-mono}). 

Next we will present a series of auxiliary results, 
which will be used later in the paper. 

Besides the notion of weak solutions given in Definition~\ref{def-sol}, 
we need to be able to uniquely solve the Dirichlet problem in smooth
cylinders. 
Given measurable nonnegative functions $u_0$ on 
$U\Subset\Rn$ 
and $g$ on the lateral boundary
$\Sigma_{t_1,t_2}=\partial U\times(t_1,t_2]$, we are interested 
in finding a weak solution $u=u(x,t)$ defined in $U_{t_1,t_2}$ that solves 
the \emph{boundary value problem}
\begin{equation}\label{eq-bvp}
\begin{cases}
\bdy_t u = \Delta u^m& \text{in }\ U_{t_1,t_2},\\
u(\,\cdot\,,t_1)=u_0&\text{in }\ U,\\
u^m=g &\text{in }\ \Sigma_{t_1,t_2}.
\end{cases}
\end{equation}
We need to define in which sense the initial condition and the lateral 
boundary data are taken. 

It is well known that for sufficiently smooth $U$, functions 
$f\in W^{1,2}(U)$ have boundary values $T_{\partial U}f$,
called traces, on the boundary $\partial U$ (see e.g.\
DiBenedetto~\cite[Theorem~18.1]{DB-real}).
Moreover, the linear trace map $T_{\partial U}$ maps $W^{1,2}(U)$ 
onto the space $W^{1/2,2}(\partial U)\subset L^2(\partial U)$,
and $T_{\partial U} f =f|_{\partial U}$ if $f \in W^{1,2}(U) \cap C(\overline{U})$.
In the time dependent context, the trace operator can be naturally 
extended into a continuous linear map
\[
T_{\Sigma_{t_1,t_2}}: L^2(t_1,t_2;W^{1,2}(U)) \longrightarrow 
L^2(t_1,t_2;W^{1/2,2}(\partial U))\subset L^2(\Sigma_{t_1,t_2}).
\]

In V\'azquez~\cite[Theorems~5.13 and 5.14]{Vazquez-book-1} 
the following result is proved, 
which addresses the problem of existence and uniqueness 
in the framework of $L^p$ spaces. 
A somewhat analogous result is proved in Alt--Luckhaus~\cite{altL}. 

\begin{thm}\label{thm-dirichlet-bvp} 
Let $m >0$ and let $U_{t_1,t_2}$ be a $C^{2,\alpha}$-cylinder.
Also let $u_0\in L^{m+1}(U)$ be nonnegative and assume that there exists 
 $\gb\in L^2(t_1,t_2;W^{1,2}(U))$ 
such that 
\[
T_{\Sigma_{t_1,t_2}}(\gb)=g
\]
and $\gb,\partial_t \gb\in L^\infty(U_{t_1,t_2})$. 
Then there exists a unique weak solution $u$ in $U_{t_1,t_2}$ such that
\begin{itemize}
\item $T_{\Sigma_{t_1,t_2}}(u^m)=g$,
\item $u(\,\cdot\,,t)\to u_0
$ in the $L^1(U)$ topology, as $t\to t_1$. 
\end{itemize}
Finally, the comparison principle applies to these solutions\/\textup{:}
if $u$ and $\hat u$ are weak solutions corresponding to $g, u_0$ and 
$\gh, \uh_0$, respectively, with
$u_0\le\hat u_0$ a.e.\ in $U$ and $g\le\hat g$ a.e.\ in 
$\Sigma_{t_1,t_2}$, 
then $u\le\hat u$ a.e.\ in $U_{t_1,t_2}$. 
\end{thm}

We also need to consider the existence of \emph{continuous} solutions. 
Under this point of view, if we assign 
continuous data on the whole 
parabolic boundary, we have the following result.

\begin{thm}\label{thm-cont-exist} 
Let $m >0$ and let $U_{t_1,t_2}$ be a $C^{1,\be}$-cylinder,
where $\be=\frac{m-1}{m+1}$ if $m>1$ and  $\be >0$ if $0<m \le 1$.
Also let $h\in C(\partial_p U_{t_1,t_2})$ be nonnegative. 
Then there is a unique function 
$u \in C(\overline{U}_{t_1,t_2})$ 
that is parabolic in $U_{t_1,t_2}$ 
and takes the boundary values $u=h$ on the parabolic boundary 
$\partial_p U_{t_1,t_2}$. 

Moreover, 
 if $h \le h' \in C(\partial_p U_{t_1,t_2})$ and 
$u'\in C(\overline{U}_{t_1,t_2})$ is the unique function corresponding to $h'$ as above,
then $u \le u'$ in $U_{t_1,t_2}$.
\end{thm}

Variations of this second boundary value problem have been widely studied.
Aronson--Peletier~\cite{aronson-peletier} and 
Gilding--Peletier~\cite{gilding-peletier} proved the unique existence as here,
provided $U_{t_1,t_2}$ is a $C^{2,\alpha}$-cylinder, $m>1$, and one has 
\emph{homogeneous conditions} $h=0$ on the lateral boundary.
We need this unique existence 
for general boundary conditions, 
in which case the result can be seen as a consequence of 
Abdulla~\cite{abdulla-00}, \cite{abdulla-05}, 
DiBenedetto~\cite{DiBe1986} and Vespri~\cite{vespri}; 
see the comments in the proof below.

\begin{proof}
In \cite{abdulla-00} and \cite{abdulla-05}, Abdulla studies the unique solvability of 
the Dirichlet problem with continuous boundary data $h$ in a general 
(not necessarily cylindrical) open set $\Theta\subset\R^{n+1}$.
In particular, conditions $\mathcal A$ and $\mathcal B$ of 
\cite[Theorem~2.1]{abdulla-00} ensure existence, 
whereas condition $\mathcal M$ of \cite[Theorem~2.2]{abdulla-05} 
ensures uniqueness. 
When $\Theta$ is a cylinder, 
$\mathcal A$ and $\mathcal B$ coincide.
It is not hard to verify that if $U$ is a bounded 
open $C^{1,\be}$-smooth set,
with $\be$ as above, 
then at every point of the parabolic boundary $\partial_p U_{t_1,t_2}$ 
conditions $\mathcal B$ and $\mathcal M$ are satisfied, yielding the 
unique existence of a suitable solution in $C(\overline{U}_{t_1,t_2})$. 

As a matter of fact, Abdulla 
uses a definition of solution, which is weaker than Definition~\ref{def-sol}. 
However, the existence of a function $u\in C(\overline{U}_{t_1,t_2})$, 
that is parabolic (in our sense)
in $U_{t_1,t_2}$ and takes the boundary values $u=h$  on the parabolic 
boundary $\partial_p U_{t_1,t_2}$, follows from 
DiBenedetto~\cite[Remark~1.2]{DiBe1986}
(for $m>1$) and Vespri~\cite[Theorem~1.1 and Remarks~(a) and~(d)]{vespri}
(for $0<m \le 1$). 
Using integration by parts it can be shown that this parabolic function 
is a solution in the sense of Abdulla. 

Since solutions in the sense of Abdulla are unique, it follows that 
the parabolic function provided by \cite{DiBe1986} 
or \cite{vespri} is the unique 
continuous weak solution of the boundary value problem.

Finally, the inequality 
$u \le u'$ follows from \cite[Theorem~2.3]{abdulla-05}.
\end{proof}

Having considered existence and uniqueness, we also need the following 
stability result from Abdulla~\cite[Corollary~2.3]{abdulla-05}.

\begin{thm} \label{thm-stability}
Let $m >0$ and let $U_{t_1,t_2}$ be a $C^{1,\be}$-cylinder,
where $\be=\frac{m-1}{m+1}$ if $m>1$ and  $\be >0$ if $0<m \le 1$.
Also let
$h_j \in C(\partial_p U_{t_1,t_2})$ be nonnegative,
and let $u_j \in C(\overline{U}_{t_1,t_2})$ be the corresponding
solutions given by Theorem~\ref{thm-cont-exist}, $j=0,1,2,\ldots$.
If $\sup_{\partial_p U_{t_1,t_2}} |h_j-h_0| \to 0$ as $j \to \infty$,
then $u_j$ tends to $u_0$ locally uniformly in $U\times(t_1,t_2]$
as $j\to\infty$.
\end{thm}

We proceed by stating a comparison principle for sub- and supersolutions 
in cylinders. 
It was first proved in $\R^{1+1}$ by Aronson--Crandall--Peletier~\cite{aronson-82},
and in $\R^{n+1}$ 
by Dahlberg--Kenig~\cite{Dahlberg-Kenig-84}, \cite{Dahlberg-Kenig-86}, 
\cite{Dahlberg-Kenig-88}.
A further and somewhat different statement of the comparison principle is 
given in Abdulla~\cite[Theorem~2.3]{abdulla-05}. 
For the proof of the following statement, we refer the reader to 
Daskalopoulos--Kenig~\cite[pp.\ 10--12]{Daskalopoulos-Kenig-book} and 
V\'azquez~\cite[pp.\ 132--134]{Vazquez-book-1}. 

\begin{prop}
\label{prop-comp-sub-supersoln}
\textup{(Comparison principle for sub- and supersolutions)}
Let $m >0$ and let $U_{t_1,t_2}$ be a $C^2$-cylinder.
Suppose $u$ and $v$ are a super- and a subsolution  
in $U_{t_1,t_2}$, respectively, such that $u^m, v^m\in L^2(U_{t_1,t_2})$. 
Assume furthermore that 
\[
    \begin{cases}
    (v^m-u^m)_\limplus(\,\cdot\,, t) \in W_0^{1,2}(U) &
        \text{for a.e. } t \in (t_1, t_2), \\
    (v-u)_\limplus(x, t_1)=0 & \text{for a.e. } x \in U.
    \end{cases}
\]
Then $0 \le v \le u$ a.e.\ in $U_{t_1,t_2}$. 
\end{prop}

Proposition~\ref{prop-comp-sub-supersoln} is the first of many
comparison principles in this paper. 
This is the only one between sub- and supersolutions, but we will have several 
different parabolic (Theorems~\ref{thm-comp-KLL}, 
\ref{thm-para-comp-princ} and~\ref{thm-para-comp-princ-union-cyl})
and one elliptic-type (Theorem~\ref{thm-comp-princ})
comparison principles for sub- and superparabolic functions.
In addition, sub- and superparabolic functions
will be defined using yet another  type of comparison principle,
for which we also have alternative versions in 
Proposition~\ref{prop-def-alt-superparabolic}
and Remark~\ref{rmk-def-superparabolic}.

\section{Definition of superparabolic functions}\label{S:Superparabolic}

\begin{definition}\label{def:superparabolic}
A function $u:{\Th}\subset\Rno\rightarrow [0,\infty]$
is \emph{superparabolic} if
\begin{enumerate}
\renewcommand{\theenumi}{\textup{(\roman{enumi})}}%
\item \label{i} $u$ is lower semicontinuous;
\item \label{ii} $u$ is finite in a dense subset of ${\Th}$;
\item \label{iii}
$u$ satisfies the following comparison principle on each 
  $C^{2,\alp}$-cylinder $U_{t_1,t_2}\Subset{\Th}$:
If $h\in C(\overline{U}_{t_1,t_2})$ is parabolic in
  $U_{t_1,t_2}$ and 
  $h\leq u$ on $\partial_p U_{t_1,t_2}$, then $h\leq u$ in  
$U_{t_1,t_2}$.
\end{enumerate}

That $v$ is \emph{subparabolic} is defined analogously,
except that $v:\Th\to [0,\infty)$ is upper semicontinuous 
and the inequalities are reversed, i.e.\
 we require that if  $h\ge v$ on $\partial_p U_{t_1,t_2}$, 
then $h\ge v$ in 
$U_{t_1,t_2}$.
\end{definition}

Note that as with sub- and supersolutions we implicitly assume
that sub- and superparabolic functions are nonnegative in this paper. 

In Kinnunen--Lindqvist~\cite{Kinnunen-Lindqvist:2008},  
Kinnunen--Lindqvist--Lukkari~\cite{Kinnunen-Lindqvist-Lukkari:2013}
and Avelin--Luk\-kari~\cite{AvelinL2015}
they require \ref{iii} in Definition~\ref{def:superparabolic} 
to hold for arbitrary
compactly contained cylinders $U_{t_1,t_2}\Subset{\Th}$.
(In \cite{Kinnunen-Lindqvist:2008} and \cite{Kinnunen-Lindqvist-Lukkari:2013}
they use the name ``viscosity supersolution'' instead of superparabolic,
while in \cite{AvelinL2015} they call them
``semicontinuous supersolutions''.)

One of our first aims is to show that our Definition~\ref{def:superparabolic}
is equivalent to the definition
in 
\cite{AvelinL2015}, \cite{Kinnunen-Lindqvist:2008}
and~\cite{Kinnunen-Lindqvist-Lukkari:2013},
 when $m \ge 1$.
This will take some effort and will only be completed at the end of this section.
The reason for our unorthodox definition is that we
want to establish Theorem~\ref{thm-lsc-supersoln},
which we have not been able to prove without using our definition.
Once Theorem~\ref{thm-lsc-supersoln} has been deduced we are
able to show that our definition of sub- and superparabolic
functions is equivalent to the one in 
\cite{AvelinL2015}, \cite{Kinnunen-Lindqvist:2008}
and~\cite{Kinnunen-Lindqvist-Lukkari:2013},
when $m \ge 1$, see Remark~\ref{rmk-def-superparabolic}.

The following consequences of the definition of sub- and superparabolicity
are almost immediate, we leave the proof to the reader.

\begin{lem} \label{lem-min-superparabolic} 
The following hold for all $m>0$\/{\rm:}
\begin{enumerate}
\item
if $u$ and $v$ are superparabolic, then $\min\{u,v\}$ is superparabolic\/{\rm;}
\item
if $u$ is finite in a dense set, then $u$ is superparabolic
if and only if $\min\{u,k\}$ is superparabolic for $k=1,2,\dots$\/{\rm;}
\item
if $u$ and $v$ are subparabolic, then $\max\{u,v\}$ is subparabolic\/{\rm;}
\item
if $v$ is subparabolic, then $v$ is locally bounded.
\end{enumerate}
\end{lem}

For a function $u$ we define the \emph{lsc-regularization} of $u$ as
\[
u_*(\xi_0)=\essliminf_{\xi\to\xi_0}u(\xi).
\]
We also say that $u$ is \emph{lsc-regularized} if $u_*=u$.
Avelin--Lukkari~\cite{AvelinL2015} 
proved the following result.

\begin{thm}\label{thm-lsc-repr} 
Let $m \ge 1$ and let 
$u$ be a supersolution. 
Then,
\[
   u_*(x, t) = u(x, t)
\]
at all Lebesgue points of $u$ such that $u(x, t)<\infty$. 
In particular, $u_*=u$ a.e., and $u_*$ is  a lower 
semicontinuous representative of $u$.
\end{thm}

Strictly speaking, Avelin--Lukkari~\cite{AvelinL2015}  only considers $m>1$; 
the remaining case $m=1$ 
can be recovered from Kuusi~\cite{Kuusi09} assuming $p=2$.

Similarly, for a function $u$ we define the 
\emph{usc-regularization} of $u$ as
\[
u^*(\xi_0)=\esslimsup_{\xi\to\xi_0}u(\xi)
\]
and 
say that $u$ is \emph{usc-regularized} if $u^*=u$. 
We will also need  the following result.

\begin{thm}\label{thm-usc-repr} 
Let $m \ge 1$ and let 
$u$ be a subsolution.
Then,
\[
  u^*(x, t) = u(x, t)
\]
at all Lebesgue points of $u$. 
In particular, $u^*=u$ a.e., and
$u^*$ is an upper semicontinuous representative of $u$.
\end{thm}

Due to the structure of the porous medium equation, 
this 
is not a trivial consequence of 
Theorem~\ref{thm-lsc-repr}, but needs to be proved separately.
We postpone the proof of Theorem~\ref{thm-usc-repr} 
to Appendix~\ref{app-pf-usc}.

Note that we do not need to require that $u(x, t)$ is finite in 
Theorem~\ref{thm-usc-repr},
since $u$ is nonnegative 
and 
subsolutions are essentially bounded from above when $m \ge 1$; 
see Andreucci~\cite{andreucci}.

\begin{thm} \label{thm-lsc-supersoln}
Let $m \ge 1$. 
If $u$ is a  supersolution then $u_*$ is superparabolic.
Similarly, if $v$ is a  subsolution then $v^*$ is subparabolic.
\end{thm}

In a less precise form this result was stated
just after Theorem~1.1
in Avelin--Lukkari~\cite{AvelinL2015}, without proof.
We therefore provide a complete proof 
of this result, and this is also the reason
for our unorthodox definition of sub- and superparabolic functions.
Once Remark~\ref{rmk-def-superparabolic} has been
established below, it follows directly 
that Theorem~\ref{thm-lsc-supersoln} 
is also valid using the sub- and superparabolic
definition used in 
\cite{AvelinL2015}, \cite{Kinnunen-Lindqvist:2008}
and~\cite{Kinnunen-Lindqvist-Lukkari:2013}.

\begin{proof}
Assume first that $u$ is a supersolution.
By Theorem~\ref{thm-lsc-repr}, $u_*=u$ a.e., and thus
also $u_*$ is a supersolution.
We want to show that $u_*$ is superparabolic.
Condition \ref{i} follows from Theorem~\ref{thm-lsc-repr},
while \ref{ii} follows directly.
For \ref{iii}, fix a $C^{2,\alpha}$-cylinder ${U}_{t_1,t_2} \Subset \Theta$
and let $h \in C(\overline{U}_{t_1,t_2})$  be such that
it is parabolic in $U_{t_1,t_2}\Subset \Th$ and $h \le u_*$ on
$\bdy_p U_{t_1,t_2}$.

According to Definition~\ref{def-sol}, this means that
$h^m\in L^2(s_1,s_2;W^{1,2}(V))$ for every cylinder $V_{s_1,s_2}\Subset U_{t_1,t_2}$,
but this is not enough to directly apply the comparison principle
in Proposition~\ref{prop-comp-sub-supersoln}, which would require
$h^m\in L^2(t_1,t_2;W^{1,2}(U))$.
We therefore proceed as follows.

Let $\hb_j \in C^\infty(\Rno)$ and $h_j= \hb_j|_{\bdy_p U_{t_1,t_2}}$
be such that  $0 \le h_j \le h$ on $\bdy_p U_{t_1,t_2}$
and $\sup_{\bdy_p U_{t_1,t_2}} |h_j -h| \to 0$, as $j \to \infty$.
Using Theorem~\ref{thm-dirichlet-bvp},
we can extend $h_j$ so that it is a weak solution in $U_{t_1,t_2}$
which takes the boundary data $h_j$ in the sense of traces
and which satisfies
$h^m_j(\,\cdot\,,t)\in W^{1,2}(U)$ for a.e.\ $t \in (t_1,t_2)$.
By the comparison principle in
Proposition~\ref{prop-comp-sub-supersoln},
 $h_j \le u_*  $ a.e.\ in $U_{t_1,t_2}$.
Since $h_j$ is continuous, and $u_*$ is lsc-regularized,
it directly follows that $h_j \le u_*$ everywhere in $U_{t_1,t_2}$.

Moreover, since the boundary data $h_j$ are continuous,
DiBenedetto~\cite[Theorem, p.\ 421]{DiBe1986} implies that
$h_j\in C(\overline{U}_{t_1,t_2})$.
Hence $h_j$ coincides with the solution provided by
Theorem~\ref{thm-cont-exist}.
Letting $j\to\infty$, we conclude from Theorem~\ref{thm-stability}
that $h \le u_*$ everywhere in $U_{t_1,t_2}$.
Hence $u_*$ is superparabolic.
The proof for subsolutions is analogous, using Theorem~\ref{thm-usc-repr}.
\end{proof}

To establish the equivalence between
our sub- and superparabolic functions and the ones used in 
\cite{AvelinL2015}, \cite{Kinnunen-Lindqvist:2008}
and~\cite{Kinnunen-Lindqvist-Lukkari:2013},
we will also need the following parabolic
comparison principle for sub- and superparabolic functions,
which was obtained by
Kinnunen--Lindqvist--Lukkari~\cite[Theorem~3.3]{Kinnunen-Lindqvist-Lukkari:2013}.

\begin{thm} \label{thm-comp-KLL}
\textup{(Parabolic comparison principle for cylinders)}
Let $m \ge 1$ and let $U_{t_1,t_2}$ be an arbitrary cylinder 
in $\R^{n+1}$. 
Suppose that $u$ is a bounded super\-parabolic function
 and $v$ is a bounded sub\-parabolic function
in $U_{t_1,t_2}$.
Assume that
\begin{equation}  \label{eq-limsup-le-liminf}
  \limsup_{U_{t_1,t_2}  \ni (y,s)\rightarrow (x,t)} v(y,s) \le
   \liminf_{U_{t_1,t_2} \ni (y,s)\rightarrow (x,t)} u(y,s) 
\end{equation}
for all
$(x,t) \in\bdyp U_{t_1,t_2}$.
Then $v\le u$ in $U_{t_1,t_2}$.
\end{thm}

As the definition of superparabolic functions in 
\cite{Kinnunen-Lindqvist-Lukkari:2013} 
is slightly different from ours,  some comments are in order.
Since we also had difficulties understanding how they concluded
that $u \le v$ everywhere (and not just a.e.) at the end of their
proof, we seize the opportunity to provide our own proof
(based partly on the ideas in \cite{Kinnunen-Lindqvist-Lukkari:2013}). 

\begin{proof}
Without loss of generality we can assume that both $u$ and $v$ are bounded.
Using~\eqref{eq-limsup-le-liminf} and the compactness of $\bdy_p U_{t_1,t_2}$,
we can for each $\eps_j=1/j$, $j=1,2,\ldots$, find
$C^{2,\al}$-cylinders 
$U^j_{s_j,t_2}:=U^j\times(s_j,t_2) \Subset U\times (t_1,t_2]$ so that 
\[
U^1\Subset U^2\Subset \cdots \Subset 
\bigcup_{j=1}^\infty U^j = U,
\quad s_1 > s_2 > \cdots \to t_1,
\]
and 
\[
v^m \le u^m + \eps_j^m \quad 
\text{in } U_{t_1,t_2} \setm U^j_{s_j,t_2},  \ j=1,2,\ldots.
\]
Since $u$ and $v$ are lower and upper semicontinuous, respectively, 
we can also find nonnegative $\hb_j\in C^\infty(\R^n)$ such that
\[
v^m \le \hb_j^m + \eps_j^m \le u^m + \eps_j^m 
\quad \text{in } U_{t_1,t_2} \setm U^j_{s_j,t_2}, \ j=1,2,\ldots.
\]
As in the proof of Theorem~\ref{thm-lsc-supersoln}, we use
Theorem~\ref{thm-dirichlet-bvp}, together with
DiBenedetto~\cite[Theorem, p.\ 421]{DiBe1986},
to find weak solutions 
$h_j,\hh_j$ in $U^j_{s_j,t_2}$ which take the boundary 
data $\hb_j$ and $(\hb_j^m + \eps_j^m)^{1/m}$, respectively,
both in the sense of traces
and continuously on $\bdy_p U^j_{s_j,t_2}$.
The super/subparabolicity of $u$ and $v$ now yield
\begin{equation}   \label{eq-v-le-ht-h-u}
u\ge h_j \quad \text{and} \quad v\le \hh_j 
\qquad \text{in } U^j_{s_j,t_2}.
\end{equation}
If we extend $h_j$ and $\hh_j$ as $\hb_j$ and $(\hb_j^m + \eps_j^m)^{1/m}$
outside $U^j_{s_j,t_2}$, then also
\begin{equation}   \label{eq-v-le-ht-h-u-outside}
u\ge h_j \quad \text{and} \quad v\le \hh_j 
\qquad \text{in } U_{t_1,t_2} \setm U^j_{s_j,t_2}.
\end{equation}
Moreover, 
\begin{equation}   \label{eq-hh-le-h+eps}
h_j\le\hh_j \le h_j + \eps_j \text{ in } U_{t_1,t_2}\setm U^j_{s_j,t_2}
\quad \text{and} \quad h_j\le\hh_j \text{ in } U^j_{s_j,t_2},
\end{equation}
by Proposition~\ref{prop-comp-sub-supersoln}.

Now, Theorem~5.16.1 in DiBenedetto--Gianazza--Vespri~\cite{DBGV-mono}
shows that both families $\{h_j\}_{j=1}^\infty$ and $\{\hh_j\}_{j=1}^\infty$
are locally equicontinuous in $U_{t_1,t_2}$.
Hence, A\-scoli's theorem and a diagonal argument provide us with subsequences,
also denoted 
$\{h_{j}\}_{j=1}^\infty$ and $\{\hh_{j}\}_{j=1}^\infty$, which converge
locally uniformly in $U_{t_1,t_2}$ to continuous functions $h$ and $\hh$.
Clearly, $h\le\hh$ and taking limits in \eqref{eq-v-le-ht-h-u} 
and~\eqref{eq-v-le-ht-h-u-outside} yields
\begin{equation}  \label{eq-u-ge-h}
u\ge h \quad \text{and} \quad v\le \hh  \qquad \text{in } U_{t_1,t_2}.
\end{equation}
For each $j=1,2,\ldots$, Lemma~3.2 in 
Kinnunen--Lindqvist--Lukkari~\cite{Kinnunen-Lindqvist-Lukkari:2013}
implies that 
\[
\iintlim{s_j}{t_2}\int_{U^j} (\hh_j-h_j)(\hh^m_j-h^m_j) \,dx\,dt \le C\eps_j,
\]
where $C$ depends on $U$ and the bounds for $u$ and $v$, but not on $j$.
Taking into account \eqref{eq-hh-le-h+eps}, we thus conclude that
\[
0\le \iintlim{t_1}{t_2}\int_U (\hh_j-h_j)(\hh^m_j-h^m_j) \,dx\,dt \le C\eps_j.
\]
Since $h_{j}\to h$ and $\hh_{j}\to \hh$ in $U_{t_1,t_2}$
and all the functions are uniformly bounded, dominated convergence
implies that
\[
\iintlim{t_1}{t_2}\int_U (\hh-h)(\hh^m-h^m) \,dx\,dt =0.
\]
and hence $h=\hh$ a.e. 
Finally, the continuity of $h$ and $\hh$, together with \eqref{eq-u-ge-h},
yields $v\le\hh=h\le u$.
\end{proof}

\begin{remark}
The above proof also shows that the function $h=\hh$ is a weak solution
in $U_{t_1,t_2}$.
Indeed, the Caccioppoli inequality 
(Lemma~2.15 in Kinnunen--Lindqvist~\cite{Kinnunen-Lindqvist:2008})
shows that $|\grad h_j^{m}|$ and
$|\grad \hh_j^{m}|$ are uniformly bounded in $L^2(s,t;W^{1,2}(V))$ for
every cylinder $V_{s,t}\Subset U_{t_1,t_2}$. 
Thus, there is a weakly converging subsequence, for which the integral 
identity \eqref{eq:weak-solution} on $V_{s,t}\Subset U_{t_1,t_2}$ pertains.
\end{remark}

\begin{prop} \label{prop-def-alt-superparabolic}
Let $m \ge 1$. 
If $u$ is superparabolic in $\Th$,
then it satisfies the following comparison principle on each 
  cylinder $U_{t_1,t_2}\Subset{\Th}$\textup{:}
If $h \in C(\overline{U}_{t_1,t_2})$ is parabolic in
  $U_{t_1,t_2}$ and 
  $h\leq u$ on $\partial_p U_{t_1,t_2}$, then $h\leq u$ in  $U_{t_1,t_2}$.

Similarly, if $v$ is subparabolic in $\Th$, $U_{t_1,t_2}\Subset{\Th}$ 
is a cylinder, $h \in C(\overline{U}_{t_1,t_2})$ is parabolic in
$U_{t_1,t_2}$ and 
$h\ge v$ on $\partial_p U_{t_1,t_2}$, then $h\ge u$ in $U_{t_1,t_2}$.
\end{prop}

\begin{remark} \label{rmk-def-superparabolic} 
This shows that our definition of sub- and superparabolic functions
is equivalent to the one used in 
Kinnunen--Lindqvist~\cite{Kinnunen-Lindqvist:2008},
Kinnunen--Lindqvist--Lukkari~\cite{Kinnunen-Lindqvist-Lukkari:2013}
and  Avelin--Lukkari~\cite{AvelinL2015}.
It also follows from Theorem~\ref{thm-para-comp-princ-union-cyl} below,
that one can equivalently assume that the comparison principle 
holds for all compactly contained finite unions of cylinders;
this equivalence was also pointed out in 
\cite[p.\ 147]{Kinnunen-Lindqvist:2008}.

Whether it is equivalent to just assuming that the comparison principle holds
for space-time boxes 
$(a_1,b_1)\times\ldots\times(a_n,b_n)\times(t_1,t_2)$
is an open problem.
Such an equivalence is known to hold for the  \p-parabolic 
equation \eqref{eq-p-para},
see Korte--Kuusi--Parviainen~\cite[Corollary~4.7]{KoKuPa10}.
\end{remark}

\begin{proof}[Proof of Proposition~\ref{prop-def-alt-superparabolic}] 
Let $u$ be superparabolic and let $U_{t_1,t_2}\Subset{\Th}$ be a cylinder.
By Theorem~\ref{thm-lsc-supersoln}, $h$ is subparabolic in $U_{t_1,t_2}$. 
Since $h$ is continuous on $\overline{U}_{t_1,t_2}$, it is also bounded. 
By Lemma~\ref{lem-min-superparabolic}, we have that 
$\tilde u = \min\{u, \max_{\overline{U}_{t_1,t_2}} h\}$ 
is a bounded superparabolic function. 
We can thus apply the comparison principle in Theorem~\ref{thm-comp-KLL}
to conclude that $h \le \tilde u \le u$ in $U_{t_1,t_2}$.

The proof for the subparabolic case is similar.
\end{proof}

\section{Further results on superparabolic functions}\label{S:Further-Superparabolic}

We continue with a few
more results on superparabolic functions that will be needed later on.

The following deep result completes the relation between
superparabolic functions and supersolutions.
In particular, a bounded function is superparabolic if
and only if it is an lsc-regularized supersolution.

\begin{thm} 
\label{thm-essliminf}
\textup{(Kinnunen--Lindqvist~\cite[Theorems~3.2 and~6.2]{Kinnunen-Lindqvist:2008})}
Let $m \ge 1$ and $u$ be superparabolic.
Then the following are true\/\textup{:}
\begin{enumerate}
\item 
$u$ is lsc-regularized, and moreover
\[
   u(x,t)=\essliminf_{\substack{(y,s) \to (x,t) \\ s < t}}u(y,s);
\]
\item \label{jj-super}
if $u$ is locally bounded, then $u$ is a supersolution.
\end{enumerate}
\end{thm}

Theorems~3.2 and~6.2 of \cite{Kinnunen-Lindqvist:2008} rely on the
results about the obstacle problem for the porous medium equation
discussed in Lemma~2.18 of the same paper. The main arguments are just
sketched, and the interested reader is referred elsewhere for the
details. Recently, the obstacle problem for the porous medium equation
has been extensively studied in
B\"ogelein--Lukkari--Scheven~\cite{Bogelein:2015} in a rather general
framework, and it is not hard to check that
\cite[Lemma~2.18]{Kinnunen-Lindqvist:2008} can be considered as a
special case of \cite[Theorem~2.6 and Corollary~2.8]{Bogelein:2015}. 

We will also need the corresponding result for subparabolic functions.

\begin{thm} 
\label{thm-esslimsup}
Let $m \ge 1$ and $u$ be subparabolic.
Then the following are true\/\textup{:}
\begin{enumerate}
\item 
$u$ is usc-regularized, and moreover
\[
   u(x,t)=\esslimsup_{\substack{(y,s) \to (x,t) \\ s < t}}u(y,s);
\]
\item \label{jj-sub}
$u$ is a subsolution.
\end{enumerate}
\end{thm}

\begin{proof}
As Kinnunen--Lindqvist~\cite{Kinnunen-Lindqvist:2008} deal
also with sign-changing functions, this follows
directly by applying Theorem~\ref{thm-essliminf} to $-u$.

For \ref{jj-sub} we do not need to assume that $u$ is locally bounded, as this
is automatic for nonnegative subparabolic functions.
\end{proof}

The following result completes the picture.

\begin{prop}
Let $m\ge 1$ and $u$ be a nonnegative function in $\Th$.
Then $u$ is parabolic if and only if it is
both sub- and superparabolic in $\Th$.
\end{prop}

Note that a parabolic function is, by Definition~\ref{def-sol},
a continuous solution, whereas
sub- and superparabolicity is defined using
the quite different Definition~\ref{def:superparabolic}.

\begin{proof}
First assume that $u$ is both sub- and superparabolic. 
Then,  $u$ is continuous.
Let $U_{t_1,t_2} \subset \Th$ be a $C^{2,\alp}$-cylinder.
By Theorem~\ref{thm-cont-exist}, there is
$h \in C(\overline{U}_{t_1,t_2})$ which is parabolic in 
$U_{t_1,t_2}$ and satisfies $h=u$ on $\bdy_p U_{t_1,t_2}$.
Since $u$ is superparabolic, $h \le u$ in $U_{t_1,t_2}$,
and as $u$ is subparabolic, $h \ge u$ in $U_{t_1,t_2}$,
i.e.\ $u=h$ in $U_{t_1,t_2}$, and in particular
$u$ is parabolic in $U_{t_1,t_2}$.
As being a solution of an equation is a local property,
$u$ is parabolic in $\Th$.

Conversely, assume that $u$ is parabolic.
Then $u$ is continuous, and thus $u=u_*=u^*$.
By Theorem~\ref{thm-lsc-supersoln}, $u$ is both sub- and superparabolic.
\end{proof}

Recall that $\Theta_T = \{ (x,t)\in\Theta: t<T\}$.

\begin{prop}  \label{prop-extend-subpar-m}
Let $m >0$.
Assume that $v$ is a subparabolic function in $\Theta_T$ 
satisfying 
$v\ge c$ for some $c\ge0$. 
Then the function
\[
w(x,t) = \begin{cases}
              v(x,t), & \text{if } (x,t)\in\Theta_T, \\
              {\displaystyle \limsup_{\Theta_T\ni(y,s)\to(x,t)} v(y,s)},      
                           & \text{if } (x,t)\in\Theta \text{ and } t=T, \\
              c, & \text{if } (x,t)\in\Theta \text{ and } t>T,
              \end{cases}
\]
is subparabolic in $\Theta$.
\end{prop}

\begin{prop}  \label{prop-extend-superpar-m}
Let $m >0$.
Assume that $v$ is a 
superparabolic function in $\Theta_T$ satisfying
$v\le M$ for some $M<\infty$. 
Then the function
\[
w(x,t) = \begin{cases}
              v(x,t), & \text{if }(x,t)\in\Theta_T, \\
              {\displaystyle \liminf_{\Theta_T\ni(y,s)\to(x,t)} v(y,s)},      
                           & \text{if } (x,t)\in\Theta \text{ and } t=T, \\
              M, & \text{if }(x,t)\in\Theta \text{ and } t>T,
              \end{cases}
\]
is superparabolic in $\Theta$.
\end{prop}

The proofs of these two results are similar, we give the proof of the
latter one.

\begin{proof}
Since $v$ is lower semicontinuous, so is $w$, and it is also bounded.
It remains to show the comparison principle.
To this end, let 
  $U_{t_1,t_2}\Subset{\Th}$ be a 
$C^{2,\alp}$-cylinder, 
and $h \in C(\overline{U}_{t_1,t_2})$ be parabolic in $U_{t_1,t_2}$ and such
that $h \le w$ on $\bdyp U_{t_1,t_2}$.
In particular $h \le M$ on $\bdyp U_{t_1,t_2}$, and thus
by (the comparison part of) Theorem~\ref{thm-cont-exist}, $h \le M$ in $U_{t_1,t_2}$.
Since $v$ is superparabolic in $\Theta_T$, we
see that $h \le w$ in $U_{t_1,t_2}$ if either $t_2 < T$ or $t_1 \ge T$.

Assume therefore that $t_1 < T \le t_2$.
Since $w=v \ge h$ in $U_{t_1,s}$ for each $s<T$, 
this holds also in $U_{t_1,T}$.
If $t_2>T$, it follows from the definition of $w$ and the continuity of $h$
that $h \le w$ in $U \times \{T\}$, and
moreover $h \le M = w $ in $U_{T,t_2}$.
\end{proof}

Using Theorem~\ref{thm-stability} 
we can obtain the following convergence result.

\begin{prop} \label{prop-inc-superparabolic}
Let $m >0$ and $u_k$ be an increasing sequence of superparabolic functions
 in $\Theta$.
If $u:=\lim_{k \to\infty} u_k$ is finite in a dense subset of $\Theta$,
then $u$ is superparabolic in $\Theta$.
\end{prop}

\begin{proof}
As the sequence is increasing, $u$ is automatically lower semicontinuous,
and thus it is only the comparison principle \ref{iii} that we need to prove.
Let   $U_{t_1,t_2}\Subset{\Th}$ be a $C^{2,\alp}$-cylinder,
and let  $h\in C(\overline{U}_{t_1,t_2})$ be parabolic in
$U_{t_1,t_2}$ and satisfy  $h\leq u$ on $\partial_p U_{t_1,t_2}$.
Let $h_j=(h|_{\partial_p U_{t_1,t_2}}-1/j)_\limplus$ on $\partial_p U_{t_1,t_2}$
and extend it to $U_{t_1,t_2}$ as the unique continuous extension
which is parabolic in $U_{t_1,t_2}$, as 
provided by Theorem~\ref{thm-cont-exist}.
It follows from the compactness and the lower semicontinuity that
for each $j$ there is $k_j$ such that $h_j \le u_{k_j}$ on $\partial_p U_{t_1,t_2}$.
As $u_{k_j}$ is superparabolic, it then follows from the definition
that $h_j \le u_{k_j} \le u$ in $U_{t_1,t_2}$.
By Theorem~\ref{thm-stability}, $h \le u$ in $U_{t_1,t_2}$.
Thus $u$ is superparabolic.
\end{proof}

For subparabolic functions we have the following result.

\begin{prop} 
Let $m >0$ and $u_k$ be a decreasing sequence of subparabolic functions 
in $\Theta$.
Then  $u:=\lim_{k\to \infty} u$ is subparabolic in $\Theta$.
\end{prop}

\begin{proof}
The proof is almost identical to the proof
of Proposition~\ref{prop-inc-superparabolic}.
However, this time the finiteness is automatic.
\end{proof}

Using this we can improve on Proposition~3.3 in 
Kinnunen--Lindqvist~\cite{Kinnunen-Lindqvist:2008} as follows 
(for nonnegative functions).

\begin{prop} 
Let $m \ge 1$.
If $u_k$ is an increasing sequence of supersolutions 
and $u:=\lim_{k \to\infty} u$ is locally bounded, 
then $u$ is a supersolution.

Similarly, if $u_k$ is a decreasing sequence of subsolutions,
then $u:=\lim_{k \to\infty} u$ is 
a subsolution.
\end{prop}

\begin{proof}
Consider first the case of supersolutions.
By Theorem~\ref{thm-lsc-repr} we may assume that $u_k$ are 
lsc-regularized.
By Theorem~\ref{thm-lsc-supersoln}, $u_k$ is superparabolic,
and thus $u$ is superparabolic, by Proposition~\ref{prop-inc-superparabolic}.
It then follows that $u$ is a supersolution by Theorem~\ref{thm-essliminf}.

The case for subsolutions is obtained similarly.
As before there is no need to assume local boundedness.
\end{proof}

We can now also conclude the following result, which we have not
seen in the literature, 
though it might be well known to experts in the field.

\begin{prop}\label{prop-weak-supersol}
Let $m \ge 1$. If $u$ and $v$ 
are supersolutions,
then so is $\min\{u,v\}$. 
Similarly, if $u$ and $v$ are subsolutions,
then so is $\max\{u,v\}$.

\end{prop}

To prove this we need the following characterization.

\begin{prop} \label{prop-char-uk}
Let $m >0$ and $u:\Theta \to [0,\infty]$ be a function 
such that $u^m \in L^{2}(t_1,t_2;W^{1,2}(U))$ whenever 
$U_{t_1,t_2} \Subset \Theta$.
Then $u$ is a supersolution if and only if $u_k:=\min\{u,k\}$ is a supersolution
for all $k=1,2,\ldots$.
\end{prop}

\begin{proof}
Assume first that $u$ is a supersolution.
Then it follows from 
DiBenedetto--Gianazza--Vespri~\cite[Lemma~3.5.1]{DBGV-mono}
that also $u_k$ is a supersolution, if $m \ge 1$.
For $0<m<1$, this was proved in a slightly different context, and
for a wider class of equations, 
in B\"ogelein--Duzaar--Gianazza~\cite[Lemma~3.1]{BDG4}.

Conversely, assume that $u_k$, $k=1,2,\ldots$, are supersolutions.
Let $U_{t_1,t_2} \Subset \Theta$ be a cylinder and $\phi \in C_0^\infty(U_{t_1,t_2})$
be nonnegative.
Then
\begin{align*}
& \iintlim{t_1}{t_2}\int_{U} { \nabla u^{m}} \cdot
\nabla\phi \, dx\,dt - \iintlim{t_1}{t_2}\int_{U} u
\partial_t \phi \, dx\,dt   \\
& \kern 10em =  
\lim_{k \to \infty} 
\iintlim{t_1}{t_2}\int_{U} { \nabla u_k^{m}} \cdot
\nabla\phi \, dx\,dt - \iintlim{t_1}{t_2}\int_{U} u_k
\partial_t\phi \, dx\,dt \ge 0,
\end{align*}
and thus $u$ is a supersolution.
\end{proof}

Using the characterization we can obtain the following consequence,
cf.\ Theorem~\ref{thm-essliminf}\,\ref{jj-super}.

\begin{prop} \label{prop-KiLind-improve}
Let $m\ge 1$ and  $u$ be superparabolic in $\Theta$.
If $u^m \in L^{2}(t_1,t_2;W^{1,2}(U))$ whenever 
$U_{t_1,t_2} \Subset \Theta$, 
then $u$ is a supersolution in $\Theta$.
\end{prop}

In general this kind of regularity does not hold for superparabolic
functions, since there are superparabolic functions which
are not supersolutions, the Barenblatt solution being
perhaps the easiest example, see 
Kinnunen--Lindqvist~\cite[p.\ 148]{Kinnunen-Lindqvist:2008}.

\begin{proof}
Let $k>0$.
By Lemma~\ref{lem-min-superparabolic}, $u_k:=\min\{u,k\}$ is superparabolic,
and hence  
a supersolution by Theorem~\ref{thm-essliminf}\,\ref{jj-super}.
It then follows from Proposition~\ref{prop-char-uk}, that $u$ is 
a supersolution.
\end{proof}

\begin{proof}[Proof of Proposition~\ref{prop-weak-supersol}]
First, assume that $u$ and $v$ are supersolutions.
By Theorem~\ref{thm-lsc-repr}, we may without loss of generality
assume that they are lsc-regularized.
It then follows from Theorem~\ref{thm-lsc-supersoln} that they are
both superparabolic, and hence by Lemma~\ref{lem-min-superparabolic},
so is $\min\{u,v\}$. 
Let $U_{t_1,t_2} \Subset \Theta$. 
As $u^m, v^m \in L^{2}(t_1,t_2;W^{1,2}(U))$,
also $\min\{u,v\}^m \in L^{2}(t_1,t_2;W^{1,2}(U))$.
Thus it follows from Proposition~\ref{prop-KiLind-improve} that
$\min\{u,v\}$ is a supersolution.

Next, we turn to the case when $u$ and $v$ are subsolutions.
By Theorem~\ref{thm-usc-repr}, we may 
assume that they  are usc-regularized.
It then follows from Theorem~\ref{thm-lsc-supersoln} that they are
both subparabolic, and, by Lemma~\ref{lem-min-superparabolic},
so is $\max\{u,v\}$.
Finally, by Theorem~\ref{thm-esslimsup}\,\ref{jj-sub},
$\max\{u,v\}$ is a subsolution.
\end{proof}

\section{Comparison principles for sub- and superparabolic functions}\label{S:Comparison}

In this section we obtain a series of different kinds of comparison principles
for sub- and superparabolic functions, which will be important later on.
Recall that one such comparison principle
has already been obtained
for cylinders when $m\ge1$ in Theorem~\ref{thm-comp-KLL}.
Note that the following two theorems do not require $m\ge1$.

\begin{theo}  \label{thm-para-comp-princ}
\textup{(Parabolic comparison principle for general sets)}
Let $m >0 $ and $\Theta$ be bounded.
Suppose that $u$ is super\-parabolic and $v$ is sub\-parabolic
in $\Theta$.
Let $T \in \R$ 
and assume that
 \begin{equation}  \label{eqn:liminf_vertailu}
  \infty \ne    \limsup_{\Theta \ni (y,s)\rightarrow (x,t)} v(y,s) <
   \liminf_{\Theta \ni (y,s)\rightarrow (x,t)} u(y,s) 
 \end{equation}
for all
$(x,t) \in\{(x,t) \in \partial\Theta : t< T\}$.
Then $v\le u$ in $\{(x,t) \in \Theta : t<T\}$.
\end{theo}

\begin{remark}
The proof of this comparison principle is very different from the proof
of the nonstrict comparison principle in cylinders in 
Theorem~\ref{thm-comp-KLL}.
Our proof is based on the proof in
Bj\"orn--Bj\"orn--Gianazza--Parviainen~\cite[Theorem~2.4]{BBGP}
for the \p-parabolic equation~\eqref{eq-p-para}.
\end{remark}

\begin{proof}[Proof of Theorem~\ref{thm-para-comp-princ}]
Let $\eps >0$ and
\[
E= \{(x,t)\in\Theta: t\le T-\eps \text{ and } v(x,t) > u(x,t)\}.
\]
By \eqref{eqn:liminf_vertailu}, together with
the  compactness of $\{(x,t)\in\bdry\Theta: t\le T-\eps\}$
and the semicontinuity of $u$ and $v$, we conclude that
$\itoverline{E}$ is a compact subset of $\Theta$.
We argue by contradiction. Assume that $E\ne \emptyset$, and
let
\[
T_0=\inf \{t : (x,t) \in E\} = \min \{t : (x,t) \in \itoverline{E}\}
\quad \text{and} \quad 
K=\{(x,t) \in \itoverline{E} : t=T_0\}.
\]
Since $K$ is compact,
we can find an open $C^{2,\alp}$-smooth set $U \subset \R^n$ 
such that
\[
      K \Subset U \times \{T_0\}\Subset \Theta,
\]
and thus also $\sigma < T_0 < \tau$ such that
\[
     \{(x,t) \in E : t \le  \tau\} \Subset U \times (\sigma,\tau]
     \Subset \Theta.
\]
In particular, the parabolic boundary
$\bdry_p U_{\sigma,\tau} \subset\Theta\setm E$, 
and hence $v \le u$ on $\bdry_p U_{\sigma,\tau}$.
(Here we could apply 
Theorem~\ref{thm-comp-KLL}, but in addition to adding the
requirement $m \ge 1$, it would also make this proof less elementary.)
Due to the semicontinuity of $u$ and $v$, there is a continuous
function $\psi$ on $\bdry_p U_{\sigma,\tau}$ such that $v \le \psi \le u$.
By Theorem~\ref{thm-cont-exist}, we can find
a function $h \in C(\overline{U}_{\sigma,\tau})$ 
which is parabolic in $U_{\sigma,\tau}$ 
and continuously attains its boundary values $h=\psi$ on $\bdry_p U_{\sigma,\tau}$.

The comparison principle in the definition of sub/super\-para\-bolic functions 
applied in $U_{\sigma,\tau}$ to $v$ and $h$, and to $u$ and $h$, 
shows that
$v\le h \le u$ in $U_{\sigma,\tau}$.  
Thus, we obtain that $U_{\sigma,\tau} \cap E = \emptyset$, and so $T_0 \ge \tau$, 
which gives a contradiction.
Hence $E$ must be empty, and letting $\eps \to 0$ concludes the proof.
\end{proof}

A direct consequence of Theorem~\ref{thm-para-comp-princ}
is the following comparison principle, which can be
considered as a sort of \emph{elliptic} version of the comparison principle,
since it does not acknowledge the parabolic boundary 
and uses all boundary points.
(To prove it just apply Theorem~\ref{thm-para-comp-princ} with a large
enough $T$.)

\begin{theo} \label{thm-comp-princ}
\textup{(Elliptic-type comparison principle)}
Let $m >0 $ and  $\Theta$ be bounded. 
Suppose that
$u$ is super\-pa\-ra\-bol\-ic and $v$ is sub\-parabolic in $\Theta$.
If
\begin{equation} \label{eq-elliptic}
   \infty \ne    \limsup_{\Theta \ni (y,s)\rightarrow (x,t)} v(y,s) <
    \liminf_{\Theta \ni (y,s)\rightarrow (x,t)} u(y,s) 
\end{equation}
for all $(x,t) \in \partial\Theta$,
then $v\leq u$ in $\Theta$.
\end{theo}

Both in Theorems~\ref{thm-para-comp-princ} and~\ref{thm-comp-princ}
we would have liked to have nonstrict comparison principles,
only assuming nonstrict inequalities in 
\eqref{eqn:liminf_vertailu} and \eqref{eq-elliptic},
but since we cannot add constants to sub/super\-para\-bolic functions,
we have not been able to achieve this.
In fact, this is a well-known problem with the comparison principle, 
and the nonstrict elliptic comparison principle
is known to be equivalent to the fundamental inequality
$\lP f \le \uP f$ between lower and upper Perron solutions, 
see Definition~\ref{def-Perron} below. 
Moreover, the parabolic-type and elliptic-type
comparison principles in Theorems~\ref{thm-para-comp-princ} 
and~\ref{thm-comp-princ}
are equivalent, since the former follows from the latter together with 
Propositions~\ref{prop-extend-subpar-m} and~\ref{prop-extend-superpar-m}.

In both comparison principles the conclusion is nonstrict,
even though the inequalities in 
\eqref{eqn:liminf_vertailu} and \eqref{eq-elliptic} are strict.
If one knew that $u_{\psi} < u_{\psi+\eps}$,
where $\psi \in C(\bdy_pU_{t_1,t_2})$ is positive 
and $u_{\psi}$ and $u_{\psi+\eps}$ are  as provided by 
Theorem~\ref{thm-cont-exist}, then a strict inequality could also be concluded, but
this seems to be one of the many open questions in the area.

Next, we extend the nonstrict parabolic comparison principle in
Theorem~\ref{thm-comp-KLL} to 
unions of bounded cylinders. 
Note that this improvement also removes the boundedness assumption
from Theorem~\ref{thm-comp-KLL}.

\begin{theo}  \label{thm-para-comp-princ-union-cyl}
\textup{(Parabolic comparison principle for unions of cylinders)}
Let $m \ge1$ and $\Theta$ be a finite union of bounded 
cylinders in $\R^{n+1}$.
Suppose that $u$ is super\-parabolic and $v$ is sub\-parabolic
in $\Theta$.
Assume that
\[ 
  \infty \ne    \limsup_{\Theta \ni (y,s)\rightarrow (x,t)} v(y,s) \le
   \liminf_{\Theta \ni (y,s)\rightarrow (x,t)} u(y,s) 
\] 
for all
$(x,t) \in\bdyp \Theta$.
Then $v\le u$ in $\Theta$.
\end{theo}

\begin{proof}
Assume that $\Theta=\bigcup_{j=1}^{N} U^{j}_{t_j,s_j}$ and
extend $u$ and $v$ to $(x,t) \in \bdyp \Theta$, by letting
\[
    v(x,t):=\limsup_{\Theta \ni (y,s)\rightarrow (x,t)} v(y,s)
    \quad \text{and} \quad
    u(x,t):=\liminf_{\Theta \ni (y,s)\rightarrow (x,t)} u(y,s).
\]
We argue by contradiction, assuming that 
$E=\{\xi \in \Theta : v(\xi) > u(\xi)\}$ is nonempty.
Let $\tau:=\inf\{t: (x,t) \in E \}$ and 
\[
S:=\{t_1, s_1, t_2, s_2, \ldots, t_N, s_N\}.
\]
We now divide the proof into three cases as follows.

\medskip
\emph{Case} 1. $\tau \notin S$ and thus 
$\min S<\tau<\max S$.
Let 
\[
t_0=\max\{ t \in S: t < \tau\} \quad \text{and} \quad 
s_0=\min\{ t \in S: t > \tau\}.
\]
Then $t_0 < \tau < s_0$,
$U_{t_0,s_0}=\{(x,t) \in \Theta : t_0 < t < s_0\}$ is a cylinder,
and $v \le u$ on $\bdyp U_{t_0,s_0}$.
Hence $v \le u$ in $U_{t_0,s_0}$ by Theorem~\ref{thm-comp-KLL}.
But this contradicts the fact that $\tau < s_0$.

\medskip

\emph{Case} 2. 
\emph{$\tau \in S$ but there is no point $(x,\tau) \in E$.}
In this case we let $t_0=\tau$ and proceed as in Case~1.

\medskip

\emph{Case} 3. \emph{$\tau \in S$ and there is at least one point 
$(x_\tau,\tau) \in E$.}
First, we show that $v$ is bounded. As $v$ is upper semicontinuous
and does not take the value $\infty$ at the compact set
$\bdyp \Theta$, there is $M<\infty$
such that $v < M$ on $\bdyp \Theta$.
It then follows from Theorem~\ref{thm-para-comp-princ} that 
$v \le M$ in $\Theta$.

Next, since $(x_\tau,\tau) \in E \subset \Theta$, we can find a 
$C^{2,\alp}$-cylinder $U_{t',\tau} \Subset \Theta$ with  $x_\tau \in U$.
Then there is a continuous 
$h : A:=\{(x,t) \in \bdy U_{t',\tau} : t < \tau\} \to \R$ such that
$v \le h \le u$ on $A$.
As $v$ is bounded, we can choose $h$ to be bounded.
We can then iterate Theorem~\ref{thm-cont-exist}
on $U_{t',\tau-1/j}$, $j=1,2,\ldots$,
to find a continuous solution, also called $h$,
in $U_{t',\tau}$ which has $h$ as continuous boundary values on $A$. 
By iterating also Theorem~\ref{thm-comp-KLL}, 
we see that  $v \le h \le u$ in $U_{t',\tau}$.
By  
DiBenedetto--Gianazza--Vespri~\cite[Theorem~5.16.1]{DBGV-mono},
$h$ has a continuous extension (also called $h$) to 
the top $U \times \{\tau\}$.
We then get from Theorems~\ref{thm-essliminf} and~\ref{thm-esslimsup} that
\[
   v(x_\tau,\tau) 
   =\esslimsup_{\substack{(x,t) \to (x_\tau,\tau) \\ t < \tau}} v(x,t)
   \le \lim_{\substack{(x,t) \to (x_\tau,\tau) \\ t < \tau}} h(x,t)
   \le \essliminf_{\substack{(x,t) \to (x_\tau,\tau) \\ t < \tau}} u(x,t)
   = u(x_\tau,\tau),
\]
which contradicts  the fact that $(x_\tau,\tau) \in E$.
\end{proof}

The following lemma is useful when constructing new superparabolic
functions.

\begin{lem} \label{lem-pasting}
\textup{(Pasting lemma)}
Let $m \ge 1 $ and  $G \subset \Theta$ be open. 
Suppose $u$ and $v$ are super\-pa\-ra\-bo\-lic in $\Theta$ and $G$,
respectively.
Let
\[
    w=\begin{cases}
     \min\{u,v\} & \text{in } G, \\
     u & \text{in } \Theta \setm G. \\
    \end{cases}
\] 
If $\overline{\{(x,t) \in G : v(x,t) < u(x,t) \}} \cap \Theta \subset G$, 
then $w$ is superparabolic in $\Theta$.
\end{lem}

This is a more restrictive pasting lemma than the one 
for \p-parabolic functions in
Bj\"orn--Bj\"orn--Gianazza--Parviainen~\cite[Lemma~2.9]{BBGP}.
As we only have a strict comparison
principle in Theorem~\ref{thm-para-comp-princ},
the proof in \cite{BBGP} does not carry over to our situation.
If however $u$ is constant, then we can obtain the full pasting lemma 
as follows.
Note that in applications the pasting lemma is often used with a 
constant ``outer'' function.

\begin{lem} \label{lem-pasting-const}
\textup{(Pasting lemma)}
Let $m \ge 1 $ and $0 \le k < \infty$. Suppose $G \subset \Theta$ is open and $v$ is super\-pa\-ra\-bo\-lic in $G$.
Define
\[
    w:=\begin{cases}
     \min\{k,v\} & \text{in } G, \\
     k & \text{in } \Theta \setm G. \\
    \end{cases}
\] 
If $w$ is lower semicontinuous, 
then $w$ is superparabolic in $\Theta$.
\end{lem}

Before proving Lemma~\ref{lem-pasting}, we first show how it can be used
to obtain Lemma~\ref{lem-pasting-const}.

\begin{proof}
Let $k_j=(k-1/j)_\limplus$ and 
\[
    w_j=\begin{cases}
     \min\{k_j,v\} & \text{in } G, \\
     k_j & \text{in } \Theta \setm G, \\
    \end{cases}
\] 
$j=1,2,\ldots$.
As $w$ is lower semicontinuous, it follows from the local
compactness of $\Theta \cap \bdy G$ that 
$\overline{\{(x,t) \in G : v(x,t) < k_j \}} \cap \Theta \subset G$.
Hence, by Lemma~\ref{lem-pasting}, $w_j$ is superparabolic.
Finally, using Proposition~\ref{prop-inc-superparabolic},
$w$ is superparabolic.
\end{proof}

\begin{proof}[Proof of Lemma~\ref{lem-pasting}]
Let us first show that $w$ is lower semicontinuous.
This is clear in $G$ and in $\Theta \setm \itoverline{G}$.
Let $\xi \in \Theta \cap \bdy G$. 
Then, by assumption, $w=u$ in a neighbourhood of $\xi$ and,
since $u$ is lower semicontinuous, we conclude that 
$w$ is lower semicontinuous at $\xi$,
i.e.\ $w$ is lower semicontinuous in $\Theta$.

Since $0 \le   w \le u$, $w$ is finite in a dense subset of $\Theta$,
and we only have to obtain the comparison principle.
Therefore, let $U_{t_1,t_2}\Subset{\Th}$ be a $C^{2,\alp}$-cylinder, 
and $h \in C(\overline{U}_{t_1,t_2})$ be parabolic in $U_{t_1,t_2}$ and such
that $h \le w$ on $\bdyp U_{t_1,t_2}$.
Since $h \le u$ on $\bdyp U_{t_1,t_2}$ and $u$ is superparabolic, 
we directly have that $h \le u$ in $U_{t_1,t_2}$.

Next, let $A=\overline{\{(x,t) \in U_{t_1, t_2}\cap G : v(x,t) < u(x,t) \}}$. 
Then $A \Subset G$ by assumption.
Thus, by compactness, we can cover $A$
by a finite number of cylinders 
$V^j_{\sigma_j,\tau_j}:=V^j \times (\sigma_j,\tau_j) \Subset G$.
Let $\Xi=\bigcup_{j=1}^m (V^j_{\sigma_j,\tau_j} \cap U_{t_1,t_2})$, which is 
a finite union of cylinders.
If $\xi \in \bdy_p\Xi$, then either $\xi \in \bdy_p U_{t_1,t_2}$ and thus
$h(\xi) \le w(\xi) \le v(\xi)$, or $\xi \in  U_{t_1,t_2} \setm A$ in which
case $h(\xi) \le u(\xi) \le v(\xi)$.
In either case $h \le v$ on $\bdy_p\Xi$.
Since $h$ is continuous it follows from the comparison principle in 
Theorem~\ref{thm-para-comp-princ-union-cyl} that $h \le v$ in $\Xi$.
Thus $h \le w$ in $U_{t_1,t_2}$ which shows that $w$ is superparabolic in $\Th$.
\end{proof}

We also need the corresponding pasting lemmas for subparabolic functions.
While these are not immediate consequences of the ones for superparabolic 
functions, the proofs are
easy modifications of the proofs for the superparabolic
pasting lemmas. We omit the details.

\begin{lem} \label{lem-pasting-sub}
\textup{(Pasting lemma)}
Let $m \ge 1 $ and  $G \subset \Theta$ be open. Suppose
$u$ and $v$ are sub\-pa\-ra\-bo\-lic in $\Theta$ and $G$,
respectively. Let
\[
    w=\begin{cases}
     \max\{u,v\} & \text{in } G, \\
     u & \text{in } \Theta \setm G. \\
    \end{cases}
\] 
If $\overline{\{(x,t) \in G : v(x,t) > u(x,t)\}} \cap \Theta \subset G$, 
then $w$ is subparabolic in $\Theta$.
\end{lem}

\begin{lem} \label{lem-pasting-const-sub}
\textup{(Pasting lemma)}
Let  $m \ge 1 $ and $0 \le k < \infty$. 
Suppose $G \subset \Theta$ is open and $v$ is sub\-pa\-ra\-bo\-lic in $G$.
Define
\[
    w=\begin{cases}
     \max\{k,v\} & \text{in } G, \\
     k & \text{in } \Theta \setm G. \\
    \end{cases}
\] 
If $w$ is upper semicontinuous, 
then $w$ is subparabolic in $\Theta$.
\end{lem}

\section{The Perron method and boundary regularity}\label{S:Boundary}

\emph{In Sections~\ref{S:Boundary}--\ref{S:reg-cylinders},
$\Theta \subset \R^{n+1}$ is always a\/ {\bf bounded} open set.}

\medskip

Now we come to the Perron method for \eqref{eq:para}.
For us it will be enough to consider Perron solutions for
bounded (and nonnegative)
functions, so for simplicity we restrict ourselves
to this case throughout the paper.

\begin{definition}   \label{def-Perron}
Given a bounded $f : \bdy \Theta \to [0,\infty)$,
let the upper class $\UU_f$ be the set of all
superparabolic  functions $u$ on $\Theta$ such that
\[ 
    \liminf_{\Theta \ni \eta \to \xi} u(\eta) \ge f(\xi) \quad \text{for all }
    \xi \in \bdy \Theta.
\] 
Define the \emph{upper Perron solution} of $f$  by
\[
    \uP f (\xi) = \inf_{u \in \UU_f}  u(\xi), \quad \xi \in \Theta.
\]
Similarly, let the lower class $\LL_f$ be the set of all
subparabolic functions $u$ on $\Theta$ which are
bounded above and such that
\[ 
\limsup_{\Theta \ni \eta \to \xi} u(\eta) \le f(\xi) \quad \text{for all }
\xi \in \bdy \Theta,
\] 
and define the \emph{lower Perron solution}  of $f$ by
\[
    \lP f (\xi) = \sup_{u \in \LL_f}  u(\xi), \quad \xi \in \Theta.
\]
If $\lP f = \uP f$, then $f$ is called \emph{resolutive}.
\end{definition}

Since we only have strict comparison principles in 
Theorems~\ref{thm-para-comp-princ} and~\ref{thm-comp-princ},
we also introduce strict Perron solutions as follows.

\begin{definition}   \label{def-S-Perron}
Given a bounded $f : \bdy \Theta \to [0,\infty)$,
let the upper class $\UUt_f$ be the set of all
superparabolic  functions $u$ on $\Theta$ such that
\[ 
    \liminf_{\Theta \ni \eta \to \xi} u(\eta) > f(\xi) \quad \text{for all }
    \xi \in \bdy \Theta.
\] 
Define the \emph{upper strict Perron solution} of $f$  by
\[
    \uS f (\xi) = \inf_{u \in \UUt_f}  u(\xi), \quad \xi \in \Theta.
\]
Similarly, let the lower class $\LLt_f$ be the set of all
subparabolic functions $u$ on $\Theta$ which are
bounded above and such that
\[
\limsup_{\Theta \ni \eta \to \xi} u(\eta) < f(\xi) \quad \text{for all }
\xi \in \bdy \Theta.
\]
Define the \emph{lower strict Perron solution}  of $f$ by
\[
    \lS f (\xi) = \sup_{u \in \LLt_f}  u(\xi), \quad \xi \in \Theta,
\]
if $\LLt_f \ne \emptyset$, and set $ \lS f \equiv 0$ if $\LLt_f = \emptyset$.
\end{definition}

Since $\LLt_f = \emptyset$ if, and only if, $f$ takes the value $0$ at 
some boundary point (the constant zero function allowed otherwise
is excluded in this case), the lower strict Perron solution is rather 
restrictive.
A possibility would have been to consider signed subparabolic functions
in the definition of $\LLt_f$, which we have refrained from since that would
lead into uncharted territory.

\begin{remark}\label{rmk-perron-domain}
Observe that the definitions of Perron solutions 
always depend on the set $\Theta$. 
To emphasize this dependence, we will at times use the notation 
$\uP_\Theta f, \lP_\Theta f, \uS_\Theta f$ and $\lS_\Theta f$,
as well as $\UU_f(\Theta)$, $\LL_f(\Theta)$, 
$\UUt_f(\Theta)$ and $\LLt_f(\Theta)$.

It follows from the elliptic-type comparison principle 
in Theorem~\ref{thm-comp-princ}
that $v\le u$ whenever $u\in \UUt_f$ and $v\in \LL_f$. 
Hence $\lS f \le \lP f\le \uS f$ and
similarly, $\lS f \le \uP f \le \uS f$. 
The inequality $\lP f \le \uP f$ is only known for finite unions of
cylinders, in which case it follows directly from 
the parabolic comparison principle in 
Theorem~\ref{thm-para-comp-princ-union-cyl}.

A key question in the theory is whether in general $\lP f = \uP f$. 
If this happens, the boundary data $f$ are called \emph{resolutive}.
A recent resolutivity result from 
Kinnunen--Lindqvist--Lukkari~\cite[Theorem~5.1]{Kinnunen-Lindqvist-Lukkari:2013}
shows that continuous functions are resolutive on general cylinders
when $m>1$.
In Theorem~\ref{thm-reg-monot-cyl} below we 
generalize this result to certain unions of cylinders. 
For $m=1$ (i.e.\ for the heat equation) resolutivity 
of continuous functions holds in arbitrary bounded open sets, see
e.g.\ Watson~\cite[Theorem~8.26]{watson}.
\end{remark}

Note that we have elliptic-type boundary conditions
on the full boundary, not just 
on the possibly smaller parabolic boundary, whenever it is defined. 
This is similar to the case of the 
\p-parabolic equation \eqref{eq-p-para} in 
Bj\"orn--Bj\"orn--Gianazza--Parviainen~\cite{BBGP}.
Nevertheless, the following result is true.
Recall that $\Theta_T = \{ (x,t)\in\Theta: t<T\}$. 

\begin{lem} \label{lem-restrict-Theta-T}
Let $m>0$, $T\in\R$ and suppose that 
$f:\bdy\Theta\to[0,\infty)$ is bounded.
Then
\[
   \uP f = \uP^T f 
\quad  \text{in }\Theta_T,
\]
where  $\uP^T f$ is the infimum of 
all superparabolic functions in $\Theta_T$ such that
\[ 
    \liminf_{\Theta_T \ni \eta \to \xi} u(\eta) \ge f(\xi) 
     \text{ {\rm(}or $>f(\xi)${\rm)}}
\quad \text{for all }\xi=(x,t) \in \bdy \Theta \text{ with } t<T.
\] 

Similar identity holds for $\lP f$, $\uS f$ and, when $f$ is 
bounded away from $0$, also for $\lS f$, with obvious modifications in the
definitions.
\end{lem}

Applying Lemma~\ref{lem-restrict-Theta-T} to both $\Theta$ and 
$\Theta_T$ immediately gives the following corollary.

\begin{cor} \label{cor-ThetaT}
Let $m>0$, $T\in\R$ and suppose that 
$f:\bdy\Theta\cup \bdy \Theta_T\to[0,\infty)$ is bounded.
Then
\[
   \uP_{\Theta_T} f = \uP_{\Theta} f, \quad 
   \uS_{\Theta_T} f = \uS_{\Theta} f \quad \text{and}\quad 
   \lP_{\Theta_T} f = \lP_{\Theta} f 
\qquad    \text{in } \Theta_T.
\]
If $f$ is, in addition, bounded away from $0$ then also
$\lS_{\Theta_T} f = \lS_{\Theta} f$ in $\Theta_T$.
\end{cor}

\begin{remark}  \label{rem-Pf-parab-bdy}
Note that the set $\{\xi=(x,t) \in \bdy \Theta: t<T\}$ in the definition
of $\uP^Tf$ is in general not compact. 

If $\Theta=U_{t_1,t_2}$ is a cylinder, then 
the parabolic boundary is included in the full boundary and 
contains the above set defining $\uP^Tf$.
Also the corresponding classes
of admissible super/sub\-parabolic functions are included in each other.
From this we conclude that the Perron solutions using only the parabolic
boundary $\bdy_p U_{t_1,t_2}$
lie between the two solutions $\uP f$ and $\uP^T f$,
 and thus coincide with them.

If $\Theta$ is a finite union of cylinders
(and thus the parabolic boundary is defined), the situation is less
clear, unless the boundary points 
not belonging to the parabolic boundary
are at the same time,
in which case the above argument applies.
\end{remark}

\begin{proof}[Proof of Lemma~\ref{lem-restrict-Theta-T}]
The inequality $\uP f \ge \uP^T f$ is obvious,
since the restrictions of functions from $\UU_f$ are admissible
in the definitions of $\uP^T f$.
Conversely, let $\eps >0$ and $u$ be admissible in the definition of
$\uP^{T} f$.
For $M= \sup_{\bdy \Theta} f +1$, let 
\[
v(x,t) = \begin{cases}
        M, & \text{if } (x,t) \in \Theta \text{ and } t>T-\eps, \\
        \min\{u(x,t),M\}, 
             & \text{if } (x,t) \in \Theta \text{ and } t \le T-\eps.
     \end{cases}
\]
By Lemma~\ref{lem-min-superparabolic} and
Proposition~\ref{prop-extend-superpar-m}, $v$ is superparabolic in $\Theta$
and thus $v \in \UU_f$.
Taking infimum over all $u$ shows that 
$\uP f(x,t) \le  \uP^{T} f(x,t)$ when $t <T-\eps$.
Letting $\eps \to 0$, yields
$ \uP f \le   \uP^{T} f$ in $\Theta_T$.
The identities for $\lP f$, $\uS f$ and $\lS f$ are proved similarly,
possibly replacing Proposition~\ref{prop-extend-superpar-m} by
Proposition~\ref{prop-extend-subpar-m}.
\end{proof}

\begin{thm}\label{thm-Perron-solution}
Let $m \ge 1$ and suppose 
that $f:\bdy\Theta\to[0,\infty)$ is bounded.
Then $\lP f$, $\uP f$, $\lS f$ and $\uS f$ are parabolic,
and in particular they are all continuous.
\end{thm}

\begin{proof}
For $\lP f$ and $\uP f$ in cylinders this is Theorem~4.6
in Kinnunen--Lindqvist--Lukkari~\cite{Kinnunen-Lindqvist-Lukkari:2013}
but as everything is local this is true in arbitrary sets, as they
in fact mention in \cite[p.~2960]{Kinnunen-Lindqvist-Lukkari:2013}.
For $m=1$, see e.g.\ Watson~\cite{watson}.

The proofs carry over essentially verbatim to $\lS f$ and $\uS f$.
\end{proof}

Since we cannot add constants to solutions of the porous medium equation, 
unlike in the elliptic and \p-parabolic cases, the boundary regularity might 
\emph{a priori} depend on the value of the boundary function at that point,
and could also be different from above and below.
We are therefore led to the following definitions.

\begin{definition}
\label{def:regular}
A boundary point $\xi_0\in \partial \Theta$ is 
\emph{upper regular} with respect to $\Theta$ if
\[
        \limsup_{\Theta \ni \xi \to \xi_0} \uP f(\xi) \le f(\xi_0)
\]
whenever $f : \partial \Theta \to (0,\infty)$ is positive and continuous.

Similarly, $\xi_0$ is 
\emph{lower regular for positive\/ \textup{(}nonnegative\/\textup{)} 
boundary data} with respect to $\Theta$ if
\begin{equation*}   
        \liminf_{\Theta \ni \xi \to \xi_0} \lP f(\xi) \ge f(\xi_0)
\end{equation*}
whenever $f : \partial \Theta \to [0,\infty)$ 
is positive (nonnegative) and continuous.

Finally, we say that $\xi_0$ is \emph{regular for positive\/ 
\textup{(}nonnegative\/\textup{)} boundary data} if it is both
upper regular and lower regular for positive (nonnegative) data.
\end{definition}

We will often
omit the explicit reference to $\Theta$, whenever no confusion may arise.
The following result is an elementary but useful tool.

\begin{prop} \label{prop-reg-S-P-equiv}
Let $m>0$ and $\xi_0 \in \partial \Theta$.
Then the following are true\/\textup{:}
\begin{enumerate}
\item \label{1-u}
If  $f: \bdy \Theta \to [0,\infty)$ is bounded and continuous at $\xi_0$,
and $\xi_0$ is upper regular, then
\[
        \limsup_{\Theta \ni \xi \to \xi_0} \uP f(\xi) 
        \le \limsup_{\Theta \ni \xi \to \xi_0} \uS f(\xi) 
        \le f(\xi_0).
\]
\item \label{1-nn}
If  $f: \bdy \Theta \to [0,\infty)$ is bounded and continuous at $\xi_0$,
and $\xi_0$ is lower regular for nonnegative boundary data, then
\[
        \liminf_{\Theta \ni \xi \to \xi_0} \lP f(\xi) 
        \ge f(\xi_0).
\]
\item \label{1-r}
If  $f: \bdy \Theta \to (0,\infty)$ is bounded, bounded 
away from $0$ and continuous at $\xi_0$,
and $\xi_0$ is lower regular for positive boundary data, then
\[
        \liminf_{\Theta \ni \xi \to \xi_0} \lP f(\xi) 
        \ge \liminf_{\Theta \ni \xi \to \xi_0} \lS f(\xi) 
        \ge f(\xi_0).
\]
\item \label{1-ur}
If  $f: \bdy \Theta \to (0,\infty)$ is bounded from above, bounded 
away from $0$ and continuous at $\xi_0$,
and $\xi_0$ is regular 
for positive boundary data, then
\[
        \lim_{\Theta \ni \xi \to \xi_0} \lS f(\xi) 
        =         \lim_{\Theta \ni \xi \to \xi_0} \lP f(\xi) 
        =         \lim_{\Theta \ni \xi \to \xi_0} \uP f(\xi) 
        =         \lim_{\Theta \ni \xi \to \xi_0} \uS f(\xi) 
        = f(\xi_0).
\]
\end{enumerate}
\end{prop}

\begin{proof}
\ref{1-nn}
Find $\ft \in C(\bdy \Theta)$ so that $0\le \ft \le f$ 
and $\ft(\xi_0)=f(\xi_0)$. Then
\[
        \liminf_{\Theta \ni \xi \to \xi_0} \lP f(\xi) 
        \ge \liminf_{\Theta \ni \xi \to \xi_0} \lP \tf(\xi) 
        \ge \ft(\xi_0)
        =  f(\xi_0).
\]

\ref{1-r}
First of all, for any $\eps>0$ small enough,
we can find $\ft \in C(\bdy \Theta)$ so that
$2\eps \le \ft \le f$ on $\bdy \Theta$ and $\ft(\xi_0)=f(\xi_0)$.
Now $\ft-\eps$ is positive on $\bdy \Theta$,
and thus
\[
        \liminf_{\Theta \ni \xi \to \xi_0} \lS f(\xi) 
        \ge \liminf_{\Theta \ni \xi \to \xi_0} \lP (\tf-\eps)(\xi) 
        \ge \ft(\xi_0)-\eps
        =  f(\xi_0)-\eps.
\]
Letting $\eps \to 0$ shows the second inequality,
while the first one follows directly from the inequality $\lP f \ge \lS f$.

\ref{1-u} This is shown in the same way as \ref{1-r},
using a continuous function $\tf\ge f+\eps$ such that 
$\tf(\xi_0)=f(\xi_0)+\eps$.

\ref{1-ur}
This follows from 
\ref{1-u} and \ref{1-r} and the inequalities
$\lS f \le \uP f$ and $\lP f \le \uS f$.
\end{proof}

The following are direct consequences of 
Proposition~\ref{prop-reg-S-P-equiv}.

\begin{enumerate}
\renewcommand{\theenumi}{\textup{(\roman{enumi})}}%
\item
Upper regularity is the same for positive
and nonnegative boundary data, which is the reason why we did not
define them separately. 
We do not know if such an equivalence holds for lower regularity.
\item
Upper regularity can equivalently be defined using
the upper strict Perron solution $\uS$.
Similarly, lower regularity for positive boundary data 
can be defined using the lower strict Perron solution $\lS$.
\end{enumerate}

It seems that upper regularity is easier to handle. 
At this point it is far from clear whether upper and lower regularity
are equivalent or not, nor if one may imply the other. 
One can also ask whether (upper/lower) regularity at one level, 
i.e.\ for one given boundary value 
$c \ge 0$ at $\xi_0$, is equivalent to regularity at other levels, 
and also if the growth of the functions plays a role for the regularity.

In the next two sections we are going to show that
regularity for positive boundary data can be characterized by
the existence of two two-parameter families of barriers, 
one family from above and one from below. 
Whether all barriers are really needed to guarantee regularity is 
far from obvious, but hopefully future research will be able
to clarify this. 

However, it may be worth to compare with the situation
for the \p-parabolic equation \eqref{eq-p-para} for which
regularity was characterized using one one-parameter family
in Bj\"orn--Bj\"orn--Gianazza--Parviainen~\cite[Theorem~3.3]{BBGP}.
(The crucial difference here necessitating two two-parameter families
instead of one one-parameter family is the fact that 
we can neither change sign nor add constants.)
In Bj\"orn--Bj\"orn--Gianazza~\cite[Proposition~1.2]{BBG}
it was shown that one barrier is not enough to characterize
regularity for the \p-parabolic equation, at least not for $p<2$.
But with one barrier one gets regularity for boundary data $f$ not
growing too fast, see \cite[Proposition~5.1]{BBG}.

For the heat equation one barrier is enough,
as was first shown by  Bauer~\cite[Theorems~30 and~31]{Bauer62} 
for general domains.

We end the section with the following result.

\begin{prop} \label{prop-scaled-eq}
Let $\xi_0 \in \bdy \Theta$ and $a>0$.
Assume that $m>0$, $m \ne 1$.
Then $\xi_0$ is upper/lower regular for positive\/ 
\textup{(}nonnegative\/\textup{)}
boundary data for the porous medium equation \eqref{eq:para}
if and only if it is upper/lower regular 
for positive\/ \textup{(}nonnegative\/\textup{)}
boundary data for the multiplied equation
\begin{equation} \label{eq-multiplied}
    \partial_t u = a \Delta u^m.
\end{equation}
\end{prop}

For $m=1$ it is well-known that this is false, by e.g.\ the 
Petrovski\u\i\ criterion~\cite{Petro1} and~\cite{Petro2}, 
and our proof breaks down in this case.

\begin{proof}
If $u$ is a sub/super\-so\-lu\-tion of 
\eqref{eq-multiplied}, 
then $v=a^{1/(m-1)} u$ is a sub/super\-solu\-tion of the porous medium equation
\eqref{eq:para}.
The whole theory can now equivalently be developed for the equation
\eqref{eq-multiplied} and the upper/lower regularity becomes equivalent.
\end{proof}

Since  the space $\Rno$ is homogeneous, one can translate the equation
and still have the same regularity. 
Thus, without loss of generality, we will sometimes assume that 
the boundary point under consideration is the origin $\xi_0=(0,0)$.

\section{Barrier characterization of upper regularity}\label{S:Upper-reg}

\begin{definition}
\label{def-barrier-upper} 
A family of functions $\wck$, with $c\in\Q_\limplus$ and $k=1,2,\ldots$, 
is an \emph{upper barrier family}
in $\Theta$ at the point  $\xi_0=(x_0,t_0)\in \partial \Theta$ 
if for each $c\in\Q_\limplus$ and $k$, 
\begin{enumerate}
 \item \label{cond-c-1/j}  
$\wck: \Theta\to [c,\infty)$ 
is superparabolic\/;
 \item \label{cond-lim-up} 
$\displaystyle \lim_{\Theta \ni \zeta\to \xi_0} \wck(\zeta) = c$;
 \item \label{cond-k} 
there is $j=j(c,k)\in\mathbf{N}$ such that 
\[ 
   \liminf_{\Theta \ni \zeta\to(x,t)} \wcj(\zeta) \ge c+k
\] 
for all $(x,t) \in \bdy \Theta$ with $|x-x_0|+|t-t_0| \ge 1/k$.
\end{enumerate}
\end{definition}

Here $\Qp=\{x \in \Q : x >0\}$.

\begin{thm} \label{thm:barrier-char-upper}
Let $m\ge1$ and $\xi_0\in \bdy \Theta$. 
Then the following are equivalent\/\textup{:}
\begin{enumerate}
\renewcommand{\theenumi}{\textup{(\arabic{enumi})}}%
\item \label{up-reg}
$\xi_0$ is upper regular\/\textup{;}
\item \label{up-bar}
there is an upper barrier family at $\xi_0$\/\textup{;}
\item \label{up-bar-Theta}
there is an upper barrier family at $\xi_0$ 
consisting of parabolic functions
satisfying~\ref{cond-k} in Definition~\ref{def-barrier-upper} 
with $(x,t) \in \bdy \Theta$ replaced by $(x,t) \in \overline{\Theta}$.
\end{enumerate}
\end{thm}

\begin{proof}
We assume, without loss of generality, that $\xi_0=(0,0)$.

\ref{up-bar} \imp \ref{up-reg}
Let $\{\wck\}$ be an upper barrier family 
at $\xi_0$. 

Assume that $f:\bdry\Theta\to(0,\infty)$ is continuous.
Let $c\in\Q_\limplus$ be such that $c>f(\xi_0)$.
Then find $k \ge \sup_{\partial \Theta} f$ such that
$f(x,t)<c$ whenever $|x|+|t| <1/k$.
Let $j=j(c,k)$ be as given in 
Definition~\ref{def-barrier-upper}\,\ref{cond-k}.
Since $\wcj \ge c$, we see that 
\[
\liminf_{\Theta \ni \zeta\to(x,t)} \wcj(\zeta)  > f(x,t)
\quad \text{for all } (x,t)\in\bdry\Theta.
\]
This implies that $\wcj \in \UU_f$, and thus 
$\uP f\le \wcj$ in $\Theta$. Consequently, 
\[
\limsup_{\Theta\ni\zeta\to\xi_0} \uP f(\zeta) 
\le \limsup_{\Theta\ni\zeta\to\xi_0} \wcj(\zeta)
= c.
\]
Since this holds for all rational $c>f(\xi_0)$, we conclude that
\[
\limsup_{\Theta\ni\zeta\to\xi_0} \uP f(\zeta) \le f(\xi_0),
\]
and thus $\xi_0$ is upper regular.

\ref{up-reg} \imp \ref{up-bar-Theta} 
Assume that $\xi\in\partial \Theta$ is upper regular. 
Given $c\in\Q_\limplus$ and $j=1,2,\ldots$, we let 
\[
\psi_{c,j}(x,t) := (c^m+j|x|^2+jbt^2)^{1/m}
\quad \text{for } (x,t)\in \R^{n+1},
\]
where $b=mc^{m-1}/{\diam \Theta}$.
Note that $\psicj\ge c$. 
Direct calculations show that 
\[
\partial_t \psi_{c,j} =\frac{2jbt}m (c^m+j|x|^2+jbt^2)^{-1+1/m} 
  \le \frac{2jbt}{mc^{m-1}}
\]
and
\begin{align*}
\Delta \psicj^m &=j \Delta |x|^2 = 2j \dvg x = 2j n\ge 0.
\end{align*}
Hence
\[
\bdry_t \psicj - \Delta \psicj^m \le 2j\biggl( \frac{bt}{mc^{m-1}}-n\biggr)
 \le0,
\]
which implies by Theorem~\ref{thm-lsc-supersoln} that $\psicj$ is subparabolic 
in $\Theta$.
By Theorem~\ref{thm-Perron-solution} the function 
\[
  \wcj := \lP\psi_{c,j}
\] 
is parabolic. 
Since $\psi_{c,j}$ is subparabolic, it belongs to $\LL_{\psi_{c,j}}$.
Therefore, by definition, we get that $\wcj \ge \psi_{c,j} \ge c$,
and so \ref{cond-c-1/j} in Definition~\ref{def-barrier-upper}
holds. 
As $\xi_0$ is upper regular, we also obtain
using Proposition~\ref{prop-reg-S-P-equiv} that
\begin{align*}
\limsup_{\Theta \ni \xi \to \xi_0}  \wcj
\le \limsup_{\Theta \ni \xi \to \xi_0} \uS\psi_{c,j} = c
\end{align*}
and thus $\lim_{\Theta \ni \xi \to \xi_0}  \wcj = c$, giving \ref{cond-lim-up}
in Definition~\ref{def-barrier-upper}.
By the form of $\wcj$ it follows that also \ref{cond-k} 
in Definition~\ref{def-barrier-upper} is satisfied,
and thus the functions $\wcj$ form an upper barrier family
of the type required in \ref{up-bar-Theta}.

\ref{up-bar-Theta} \imp \ref{up-bar}
This is trivial.
\end{proof}

The following useful restriction result
is a direct consequence of the barrier characterization in 
Theorem~\ref{thm:barrier-char-upper}\,\ref{up-bar-Theta}.

\begin{cor} \label{cor-restrict-up}
Let $m \ge 1$ and $G \subset \Theta$ be open. 
Suppose that $\xi_0 \in \bdy \Theta \cap \bdy G$.
If $\xi_0$ is upper regular 
with respect to $\Theta$,
then it is upper regular  
with respect to $G$.
\end{cor}

Another consequence of the barrier characterization is that
upper regularity is a \emph{local} property.

\begin{prop}  \label{prop-local-up}
Let $m \ge 1$ and $\xi_0 \in \bdy \Theta$, and suppose that
$B$ is a ball containing $\xi_0$.
Then
$\xi_0$ is upper regular 
with respect to $\Theta$
if and only if it is upper regular 
with respect to $B \cap\Theta$.
\end{prop}

\begin{proof}
Corollary~\ref{cor-restrict-up} shows 
that if $\xi_0$ is upper regular with respect to $\Theta$, 
then it is also upper regular with respect to $B \cap \Theta$. 
It remains to show the converse direction.

By Theorem~\ref{thm:barrier-char-upper} we have an 
upper barrier family $\wcj$ in $B \cap \Theta$. 
Let $k_0$ be large enough so that 
$1/k_0 < \dist(\xi_0, \partial B)$ 
and let $j(c,k_0)$ be as in Definition~\ref{def-barrier-upper}\,\ref{cond-k}. 
Define
\[
w'_{c,k} := 
\begin{cases}
\min\{w_{c, j(c,k)}, c+k\} &\text{in }  B \cap \Theta \\
c+k &\text{in }  \Theta \setminus B.
\end{cases}
\]
for $k \ge k_0$ and $w'_{c,k}= w'_{c,k_0}$ for $k < k_0$. 
By the pasting lemma~\ref{lem-pasting-const} the function $w'_{c,k}$ 
is superparabolic, and thus $\{w'_{c,k}\}$ is an upper barrier family 
in $\Theta$. 
This implies that $\xi_0$ is upper regular with respect to $\Theta$.
\end{proof}

\section{Barrier characterization of lower regularity for positive boundary data}\label{S:Lower-reg}

\begin{definition}
\label{def-barrier-lower-pos} 
A family of functions $\wck$, with $c\in\Q_\limplus$ and $k=1,2,\ldots$, 
is a \emph{lower barrier family for positive boundary data}
in $\Theta$ at the point  $\xi_0=(x_0,t_0)\in \partial \Theta$ 
if for each $c\in\Q_\limplus$ and $k$, 
\begin{enumerate}
 \item \label{cond-c+1/j}  
$\wck: \Theta\to [0,c]$ 
is subparabolic\/;
 \item \label{cond-lim} 
$\displaystyle \lim_{\Theta \ni \zeta\to \xi_0} \wck(\zeta)=c$;
 \item \label{cond-de} 
there is $j=j(c,k)\ge k$ such that 
\[ 
   \limsup_{\Theta \ni \zeta\to(x,t)} \wcj(\zeta) \le \frac1k
\] 
for all $(x,t) \in \bdy \Theta$ with $|x-x_0|+ |t-t_0| \ge 1/k$.
\end{enumerate}
\end{definition}

\begin{lem}  \label{lem-removable}
Let $m >0$. Assume that $u$ is a 
supersolution\/ \textup{(}subsolution\/\textup{)}
in $\Theta\setm (E\times\R)$, where $E\subset\R^n$ is a 
set of zero capacity such that $\Theta\setm (E\times\R)$ is open.
If $u^m\in L^{2}(t_1,t_2;W^{1,2}(U))$ for every cylinder 
$U_{t_1,t_2} \Subset \Theta$, 
then $u$ is a supersolution\/ \textup{(}subsolution\/\textup{)}
in $\Theta$ as well.
\end{lem}

For the definition of capacity, see~\eqref{eq-def-cap} below.

\begin{proof}
We consider supersolutions. The proof for subsolutions is similar.
Let $U_{t_1,t_2} \Subset \Theta$ and 
$\phi\in C_0^\infty(U_{t_1,t_2})$ be arbitrary.
Since $E$ has zero capacity, there exist $\eta_j\in C_0^\infty(\R^n)$ such
that $0\le\eta_j\le1$ in $\R^n$, $\eta_j=1 $ in an open neighbourhood of $E$
and $\|\eta_j\|_{W^{1,2}(\R^n)}<1/j$, $j=1,2,\ldots.$
Set $\phi_j(x,t)=(1-\eta_j(x))\phi(x,t)$. 
Then $\phi_j\in C_0^\infty(U_{t_1,t_2})\setm (E\times\R))$.
Inserting  $\bdy_t\phi_j = \bdy_t\phi - \eta_j\bdy_t\phi$ and
$\grad\phi_j = \grad\phi - (\eta_j\grad\phi+\phi\grad\eta_j)$ 
into~\eqref{eq:weak-solution} gives 
\begin{align*} 
&\iintlim{t_1}{t_2}\int_{U} { \nabla u^{m}} \cdot
\nabla\phi \, dx\,dt - \iintlim{t_1}{t_2}\int_{U} u
\partial_t\phi \, dx\,dt \\  
&\quad \ge 
\iintlim{t_1}{t_2}\int_{U} { \nabla u^{m}} \cdot 
(\eta_j\grad\phi+\phi\grad\eta_j)\, dx\,dt 
- \iintlim{t_1}{t_2}\int_{U} u \eta_j \partial_t\phi \, dx\,dt 
\end{align*} 
Since $\phi\in C_0^\infty(U_{t_1,t_2})$, there exists $M<\infty$ such that
$|\phi|, |\grad\phi|, |\bdy_t\phi|\le M$ on $U_{t_1,t_2}$ and hence the
Cauchy--Schwarz 
inequality implies that the right-hand side in the last equality
is majorized (in absolute value) by
\begin{align*} 
3M \biggl( \iintlim{t_1}{t_2}\int_{U} 
    (|\nabla u^{m}|^2 + |u|^2)\, dx\,dt \biggr)^{1/2}
\biggl( \iintlim{t_1}{t_2}\int_{U} 
    (|\nabla \eta_j|^2 + |\eta_j|^2)\, dx\,dt \biggr)^{1/2}.
\end{align*} 
By assumption, the first factor is bounded while the
last factor equals $$(t_2-t_1)^{1/2}\|\eta_j\|_{W^{1,2}(\R^n)}$$ and tends to zero
as $j\to\infty$. 
Thus, the left-hand side in \eqref{eq:weak-solution} is 
nonnegative for every 
$\phi\in C_0^\infty(U_{t_1,t_2})$, which concludes the proof.
\end{proof}

\begin{thm} \label{thm:barrier-char-lower-pos}
Let $m \ge 1$ and $\xi_0 \in \bdy \Theta$. 
Then the following are equivalent\/\textup{:}
\begin{enumerate}
\renewcommand{\theenumi}{\textup{(\arabic{enumi})}}%
\item \label{pos-reg}
$\xi_0$ is lower regular for positive boundary data\/\textup{;}
\item \label{pos-bar}
there is a lower barrier family for positive boundary data at $\xi_0$\/\textup{;}
\item \label{pos-bar-Theta}
there is a lower barrier family for positive boundary data at $\xi_0$ 
consisting of parabolic functions 
satisfying~\ref{cond-de} in Definition~\ref{def-barrier-lower-pos} 
with $(x,t) \in \bdy \Theta$ replaced by $(x,t) \in \overline{\Theta}$.
\end{enumerate}
\end{thm}

\begin{proof}
We assume, without loss of generality, that $\xi_0=(0,0)$.

The proof of \ref{pos-bar} \imp \ref{pos-reg} is 
similar to that of Theorem~\ref{thm:barrier-char-upper} and we omit the details.

\ref{pos-bar-Theta} \imp \ref{pos-bar}
This is trivial.

\ref{pos-reg} \imp \ref{pos-bar-Theta}
Assume that $\xi_0\in\partial \Theta$ is lower regular for positive
boundary data.
To construct a lower barrier family for positive boundary data at $\xi_0$,
let $0<\al<\ga<1/m$ and $d:=2+\log\diam\Theta$.
Given $c\in\Q_\limplus$ and $j=1,2,\ldots$, we define for $(x,t)\in \Theta$,
\[ 
\vcj(x,t) = \begin{cases}
\displaystyle        c^{-1/\ga} + j^\al t^2 + \frac{j}{d-\log|x|}, &x\ne0, \\
        c^{-1/\ga} + j^\al t^2, &x=0,   
            \end{cases}  
\quad \text{and} \quad \psicj = \vcj^{-\ga}.
\] 
Note that $0 <\psicj\le c$.
Assume that $x \ne 0$ for the moment.
Direct calculations show that $\bdry_t \vcj = 2j^\al t$,
\begin{align} \label{eq-grad-vcj}
\grad \vcj = \frac{jx}{|x|^2(d-\log|x|)^2}
\end{align}
and 
\begin{align*}
\Delta \vcj = \frac{j\bigl[(n-2)(d-\log|x|)+2\bigr]}{|x|^2(d-\log|x|)^3}
\ge \frac{2j}{|x|^2(d-\log|x|)^3},
\end{align*}
since $n\ge2$.
Moreover, we have
\begin{align}
\bdry_t \psicj & = -\ga \vcj^{-\ga-1} \bdry_t \vcj, \notag \\
\grad \psicj^m &= \grad \vcj^{-\ga m} 
         = -\ga m \vcj^{-\ga m-1} \grad \vcj,
\label{eq-grad-psicjm}
\end{align}
and thus
\[
\Delta \psicj^m 
= -\ga m \vcj^{-\ga m-1} \Delta \vcj 
        + \ga m (\ga m+1) \vcj^{-\ga m-2} |\grad\vcj|^2.
\] 
It follows that
\[
\bdry_t \psicj - \Delta \psicj^m
= \ga m \vcj^{-\ga m-2} \bigl[ \vcj\Delta\vcj - (\ga m+1) |\grad\vcj|^2 
- m^{-1} \vcj^{1+\ga(m-1)} \bdry_t \vcj \bigr].
\]

This will be nonnegative if 
\[ 
I:=
\frac{2j\vcj}{|x|^2(d-\log|x|)^3} - \frac{(\ga m+1)j^2}{|x|^2(d-\log|x|)^4}
- \frac2m \vcj^{1+\ga(m-1)} j^\al t \ge 0.
\] 
Since  
$\vcj\ge j/(d-\log|x|)$ and $m\ge1$,
we obtain
\begin{align*}
   I & \ge \frac{2j\vcj}{|x|^2(d-\log|x|)^3} - \frac{(\ga m+1)j\vcj}{|x|^2(d-\log|x|)^3}
- 2 \vcj^{1+\ga(m-1)} j^\al t \\
    & = \vcj \biggl(\frac{(1-\ga m)j}{|x|^2(d-\log|x|)^3} 
    - 2 \vcj^{\ga(m-1)} j^\al t \biggr).
\end{align*}
A straightforward calculation shows that $\rho \mapsto \rho^2(d-\log \rho)^3$
is increasing on $[0,\diam \Theta]$,
which together with  $\vcj\le j(c^{-1/\ga}+(\diam\Theta)^2+1)$ 
and the choice of $d$ yields
\begin{align*}
   I & \ge \vcj \biggl(\frac{(1-\ga m)j}{8(\diam \Theta)^2} - 2 
j^{\al+\ga(m-1)}(c^{-1/\ga}+(\diam\Theta)^2+1)^{\ga(m-1)} \diam\Theta\biggr).
\end{align*}

Since $\al<\ga< 1/m$, we have $\al+\ga(m-1)<1$ and
it follows that $I \ge 0$ 
for sufficiently large $j$. 
Thus, by Theorem~\ref{thm-lsc-supersoln} we see that $\psicj$ 
is superparabolic in $\{(x,t) \in \Theta : x \ne 0\}$ for such $j$. 
By \eqref{eq-grad-vcj} and \eqref{eq-grad-psicjm}, 
$\grad \psicj^m \in L^2(t_1,t_2;W^{1,2}(U))$ whenever
$U_{t_1,t_2} \Subset \Theta$. 
Hence Lemma~\ref{lem-removable} shows that
$\psicj$ is superparabolic in $\Theta$ for sufficiently large $j$.

Set $\wcj:=\uP \psi_{c,j}$. Then $\wcj$ is parabolic,
by Theorem~\ref{thm-Perron-solution}.
Since $\psi_{c,j}$ is superparabolic, it belongs to $\UU_{\psi_{c,j}}$,
and thus $\wcj \le \psi_{c,j} \le c$.
By the lower regularity of $\xi_0$ for positive boundary data
and Proposition~\ref{prop-reg-S-P-equiv}\,\ref{1-r}, we also have
\[
\liminf_{\Theta\ni\zeta\to\xi_0} \wcj(\zeta)  
\ge \liminf_{\Theta\ni\zeta\to\xi_0} \lS \psi_{c,j}(\zeta)  
=  \psi_{c,j}(\xi_0)= c,
\]
and thus $\lim_{\Theta\ni\zeta\to\xi_0} \wcj(\zeta)  =c$.
By the form of $\wcj$ it follows that \ref{cond-de} is satisfied,
and thus the functions $\wcj$ form a lower barrier family 
for positive boundary data of the type required in \ref{pos-bar-Theta}.
\end{proof}

As in the case of upper regularity we have the following two results.
The first is a direct consequence of the barrier characterization in 
Theorem~\ref{thm:barrier-char-lower-pos}\,\ref{pos-bar-Theta},
whereas the proof of the second  is 
 similar to the proof of Proposition~\ref{prop-local-up};
we omit the details.

\begin{cor} \label{cor-restrict-pos}
Let $m \ge 1$ and $G \subset \Theta$ be open. 
Suppose that $\xi_0 \in \bdy \Theta \cap \bdy G$.
If $\xi_0$ is lower regular for positive boundary data 
with respect to $\Theta$,
then it is lower regular for positive boundary data 
with respect to $G$.
\end{cor}

\begin{prop}  \label{prop-local-pos}
Let $m \ge 1$ and $\xi_0 \in \bdy \Theta$, and suppose that
 $B$ is a ball containing $\xi_0$.
Then
$\xi_0$ is lower regular for positive boundary data 
with respect to $\Theta$
if and only if it is lower regular for positive boundary data 
with respect to $B \cap\Theta$.
\end{prop}

\section{Earliest points are always regular}\label{S:Earlier}

To make the notation easier, we consider regularity of the origin.
Let
\[
   \Theta_\limminus:=\{(x,t) \in \Theta: t <0\}
   \quad \text{and} \quad
   \Theta_\limplus:=\{(x,t) \in \Theta: t >0\}.
\]

\begin{prop}  \label{prop-earliest-reg-below-new}
Let $m\ge 1$ and $\xi_0=(0,0)\in\bdry \Theta$.
If $\xi_0\notin\bdry \Theta_\limminus$,
which in particular holds if $\Theta_\limminus$ is empty, 
then $\xi_0$ is regular for positive boundary data.
\end{prop}

\begin{proof}
By Propositions~\ref{prop-local-up} and~\ref{prop-local-pos}
we may assume that $\Theta_\limminus=\emptyset$.

First, we turn to upper regularity and let $c \in \Qp$
be arbitrary.
Let
\[
      \wcj := (c^m+j|x|^2+j^{2m-1}t)^{1/m}
      \quad \text{and} \quad
      \de=\max\{\diam \Theta,1\}.
\]
Then  
\[
   \partial_t \wcj = \frac{j^{2m-1}}{m} (c^m+j|x|^2+j^{2m-1}t)^{1/m-1}
   \ge \frac{j^{2m-1}}{m} (c^m+2j^{2m-1} \de^2)^{1/m-1}
\]
and $\Delta \wcj^m = j \Delta |x|^2 = 2jn$.
We want to have
\[
    \partial_t \wcj - \Delta \wcj^m \ge 
    \frac{j^{2m-1}}{m} (c^m+2j^{2m-1} \de^2)^{1/m-1} -2jn \ge 0,
\]
which is equivalent to 
\[
    c^m + 2j^{2m-1} \de^2 
    \le \biggl(\frac{2nm}{j^{2m-2}}\biggr)^{m/(1-m)}
    =  \frac{j^{2m}}{(2nm)^{m/(m-1)}}
\]
and this happens if $j$ is large enough. 
Thus, for such $j$, $\wcj$ is superparabolic, 
by Theorem~\ref{thm-lsc-supersoln}.
It now follows that $\{\wcj\}$ is an upper barrier family,
and thus $\xi_0$ is upper regular.

For lower regularity, let again $c \in \Qp$
be arbitrary.
This time let
\[
      \vcj := (c^m-j|x|^2-jat)_\limplus^{1/m}
      \quad \text{and} \quad
      \Theta_{c,j}:=\{(x,t) \in \Theta : \vcj(x,t) > 0\},
\]
where $a=2nmc^{m-1}$.
In $\Theta_{c,j}$ we have  
\[
   \partial_t \vcj = -\frac{ja}{m} (c^m-j|x|^2-jat)^{1/m-1}
   \le -\frac{ja}{m} c^{1-m}=-2jn
\]
and $\Delta \vcj^m = -j \Delta |x|^2 = -2jn$.
Hence $\partial_t \vcj- \Delta \vcj^m \le 0$
and $\vcj$ is subparabolic in $\Theta_{c,j}$, 
by Theorem~\ref{thm-lsc-supersoln}.
Lemma~\ref{lem-pasting-const-sub} shows that $\vcj$ is
subparabolic in $\Theta$.
Hence, it follows that $\{\vcj\}$ is a lower
barrier family for positive boundary data
and thus by Theorem~\ref{thm:barrier-char-lower-pos},
$\xi_0=(0,0)$ is lower regular for positive boundary data.
\end{proof}

\begin{remark}
Using
the family $\{\vcj\}$ above one can show that $\xi_0=(0,0)$
is lower regular for nonnegative data, in a similar
way as the proof of \ref{pos-bar} \imp \ref{pos-reg}
in Theorem~\ref{thm:barrier-char-lower-pos}. 
We do not aim at developing 
the general theory
of lower regular points for nonnegative data here.
\end{remark}

\section{Independence of the future}\label{S:Future}

The next result shows that regularity is independent of the future.

\begin{thm} \label{thm-indep-future}
Assume $m\ge1$ and $\xi_0=(0,0) \in \bdy \Theta$.
Then $\xi_0$ is upper regular\/ 
\textup{(}lower regular for positive boundary data\/\textup{)}
 with respect to
$\Theta$ if and only if either $\xi_0 \notin \bdy \Theta_\limminus$
or $\xi_0$ is upper regular\/
\textup{(}lower regular for positive boundary data\/\textup{)}
with respect to $\Theta_\limminus$.
\end{thm}

\begin{proof}
We  consider first the upper regularity.

If $\xi_0$ is upper regular with respect to $\Theta$,
then either $\xi_0 \notin \bdy \Theta_\limminus$ or
$\xi_0$ is upper regular with respect to $\Theta_\limminus$,
by Corollary~\ref{cor-restrict-up}.

As for the converse, if $\xi_0 \notin  \bdy \Theta_\limminus$,
then $\xi_0$ is regular by 
Proposition~\ref{prop-earliest-reg-below-new}.
Thus it remains to consider the case when
$\xi_0$ is upper regular with respect to $\Theta_\limminus$.

Given $c\in\Q_\limplus$ and $j=1,2,\ldots$, we define for $(x,t)\in
\Theta$
the function
\[
\psi_{c,j} := (c^m+j|x|^2+jbt^2)^{1/m}
\]
with $b= mc^{m-1}/{\diam \Theta}$.
By the proof of Theorem~\ref{thm:barrier-char-upper}
we know that $\psi_{c,j}$ is subparabolic in $\Theta$.

Let $\wcj=\lP_{\Theta} \psicj$. We 
want to show that $\{w_{c,j}\}$ is an upper barrier 
family at $\xi_0$ with respect to $\Theta$.

As $\psicj$ is subparabolic and continuous,
$\uS_{\Theta} \psicj \ge \wcj \ge \psicj$ in $\Theta$.
By the upper regularity of $\xi_0$ with respect to $\Theta_\limminus$ 
as well as by Corollary~\ref{cor-ThetaT} 
and Proposition~\ref{prop-reg-S-P-equiv}\,\ref{1-u}, we see that
\[
   \limsup_{\Theta_\limminus \ni \zeta \to \xi_0} \wcj(\zeta)
   \le \limsup_{\Theta_\limminus \ni \zeta \to \xi_0} \uS_{\Theta} \psicj(\zeta)
   = \limsup_{\Theta_\limminus \ni \zeta \to \xi_0} \uS_{\Theta_\limminus} \psicj(\zeta)
   \le \psicj(\xi_0).
\]
Since $\wcj \ge \psicj$, we obtain
$$\lim_{\Theta_\limminus  \ni \zeta \to \xi_0} \wcj(\zeta)
   = \psicj(\xi_0).$$
Moreover, by the continuity of $\wcj$ in $\Theta$ we have
\begin{equation} \label{eq-Th-minus}
\lim_{\Theta \setm \Theta_\limplus \ni \zeta \to \xi_0} \wcj(\zeta)
   = \psicj(\xi_0).
\end{equation}
Let next
\[
    \phi = \begin{cases}
      \psicj & \text{in } \bdy \Theta \cap \bdy \Theta_\limplus, \\
      \wcj & \text{in } \Theta \cap \bdy \Theta_\limplus.
      \end{cases}
\]
Then $\phi$ is continuous at $\xi_0$, by \eqref{eq-Th-minus}.

If $u \in \LL_{\psicj}(\Theta)$, then 
$u|_{\Theta_\limplus} \in \LL_\phi(\Theta_\limplus)$,
by the definitions of $\phi$ and $\wcj$. 
Hence
\[
    \wcj = \lP_{\Theta} \psicj \le \lP_{\Theta_\limplus} \phi
    \quad \text{in } \Theta_\limplus.
\] 
If $v \in \LLt_\phi(\Theta_\limplus)$ then,
since $\wcj \ge \psicj$ in $\Theta$, we have
\[
   \limsup_{\Theta_\limplus \ni \zeta \to \xi} v(\zeta) 
   < \phi (\xi) 
   \le    \liminf_{\Theta_\limplus \ni \zeta \to \xi} \wcj(\zeta) 
   \quad \text{for all } \xi \in \bdy \Theta_\limplus.
\]
Thus, by Theorem~\ref{thm-comp-princ},
$   v \le \wcj$ in $ \Theta_\limplus$.
Taking supremum over all $v \in \LLt_\phi(\Theta_\limplus)$ shows that
\[
   \lS_{\Theta_\limplus} \phi
    \le \wcj  \le \lP_{\Theta_\limplus} \phi
    \quad \text{in } \Theta_\limplus.
\]
By Proposition~\ref{prop-earliest-reg-below-new}, the initial
points are always regular. 
Thus using also 
Proposition~\ref{prop-reg-S-P-equiv}\,\ref{1-ur}, we
see that
\[
    \lim_{\Theta_\limplus \ni \zeta \to \xi_0} \lS_{\Theta_\limplus} \phi (\zeta)
    =\lim_{\Theta_\limplus \ni \zeta \to \xi_0} \lP_{\Theta_\limplus} \phi (\zeta)
    =    \psicj(\xi_0).
\]
Therefore, we obtain
\[   
\lim_{\Theta_\limplus \ni \zeta \to \xi_0} \wcj(\zeta)= \phi(\xi_0),
\]
which together with \eqref{eq-Th-minus} shows that
\[
    \lim_{\Theta \ni \zeta \to \xi_0} \wcj (\zeta)= \psicj(\xi_0).
\]

All this together now allows us to conclude that $\{\wcj\}$ is indeed an
upper barrier family at $\xi_0$ with respect to $\Theta$.
Thus, by Theorem~\ref{thm:barrier-char-upper},
$\xi_0$ is upper regular with respect to $\Theta$.

Finally, we turn to the lower regularity (for positive boundary data).
The proof is similar to the upper regularity case above,
and the first part is analogous.
For the main part, this time we
let $0<\al<\ga<1/m$, $d:=2+\log\diam\Theta$ 
and define
\[ 
\vcj(x,t) = \begin{cases}
\displaystyle        c^{-1/\ga} + j^\al t^2 + \frac{j}{d-\log|x|}, &x\ne0, \\
        c^{-1/\ga} + j^\al t^2, &x=0,
            \end{cases}  
\quad \text{and} \quad \psicj = \vcj^{-\ga}.
\] 
This time $\psicj$ is superparabolic if $j \ge j_0(c)$,
by the proof of Theorem~\ref{thm:barrier-char-lower-pos}.
The rest of the proof is the same as for upper regularity,
with the direction reversed, and using 
Proposition~\ref{prop-reg-S-P-equiv}\,\ref{1-r} 
and Theorem~\ref{thm:barrier-char-lower-pos} at appropriate places.
We omit the details.
\end{proof}

\section{Regularity of cylinders}\label{S:reg-cylinders}

In this section, we will show that the boundary regularity for
the boundary value problem~\eqref{eq-bvp} 
for the porous medium equation 
in a cylinder is determined by the elliptic regularity of the corresponding
spatial set. 
For this reason, we recall the concept of capacity and the \emph{elliptic 
Wiener criterion}.

For a bounded
set $E\subset\R^n$, we define the \emph{capacity} of $E$ as 
\begin{equation}   \label{eq-def-cap}
\capacity(E) := \inf_u \int_V (|\nabla u|^2 + |u|^2) \, dx,
\end{equation}
where the infimum is taken over all $u \in C_0^\infty(\R^n)$
such that $u\ge1$ in a neighbourhood of $E$.

With this definition, the Wiener criterion~\cite{Wiener24} characterizes 
the \emph{regular} boundary points of a bounded open set $U\subset\R^n$, 
i.e.\ those boundary points $x_0\in\bdy U$ 
at which every solution of the elliptic boundary value problem
\[
   \begin{cases}
     \Delta u = 0 & \text{in }U, \\
        u = g\in C(\bdy U) & \text{on } \partial U,
     \end{cases}
\]
attains its continuous boundary values $g$. 
We call these points \emph{elliptic regular}.
More precisely, $x_0\in\bdy U$ is elliptic regular if and only 
if the complement of $U$ is thick at $x_0$, i.e.\ if
\begin{equation}   \label{eq-wiener}
\int_{0}^1 \frac{\capacity( 
B(x_0,r)\setm U)}{r^{n-2}} \,\frac{dr}{r} = \infty,
\end{equation}
see Wiener~\cite{Wiener24},
Littman--Stampacchia--Weinberger~\cite[Theorem~9.2]{LiStWe}
or Mal\'y--Ziemer~\cite[Theorem~4.24]{MaZi}. 

It is well known, and rather straightforward, that the Wiener 
condition~\eqref{eq-wiener} holds e.g.\ if $x_0$ is not a point of
density of $U$ or if $U$ satisfies the following porosity condition: 
There is $c>0$ and a sequence $r_k\to0$ such that for every $k=1,2,\ldots$,
the set $B(x_0,r_k)\setm U$ contains a ball of radius $cr_k$.
In particular, the famous cone and corkscrew conditions are sufficient
for boundary regularity.
Thus, our results in this section apply to a much larger class of (unions of)
cylinders than 
the ones considered in Abdulla~\cite{abdulla-00} and~\cite{abdulla-05}.

\begin{thm} \label{thm-reg-cylinder}
Let $m\ge1$ and $\Theta:=U_{t_1,t_2}$.
Also let $\xi_0=(x_0,t_0) \in \bdy \Theta$, where $t_1 < t_0 \le t_2$.
Then $\xi_0$ is regular for
positive boundary data if and only if $x_0$ 
is elliptic regular with respect to $U$.
\end{thm}

Note that by the main result in 
Kinnunen--Lindqvist--Lukkari~\cite{Kinnunen-Lindqvist-Lukkari:2013},
continuous functions are resolutive for cylinders 
(see also Remark~\ref{rmk-perron-domain}).
Also, if $t_0=t_1$ then $\xi_0$, being an earliest point, 
is always regular for positive boundary data, 
by Proposition~\ref{prop-earliest-reg-below-new}.

\begin{proof}
Without loss of generality we may assume that $\xi_0=(0,0)$,
and thus that $t_1 < 0 \le t_2$.

Assume first that $x_0$ is elliptic regular with respect to $U$.
Let $\phi(x)=|x|$, $x \in \R^n$ and let $v$ be the unique 
classical solution of
\[
    \begin{cases}
    \Delta v = -1  \text{ in } U, \\
    v-\phi  \in W^{1,2}_0(U),
     \end{cases}
\]
which exists by Theorem~4.3 in Gilbarg--Trudinger~\cite{GilbargTrudinger}.
Then $v$ is superharmonic in $U$.
Since $\phi$ is subharmonic (by a straightforward calculation), 
we have $v \ge \phi$ a.e., by Lemma~3.18 in
Heinonen--Kilpel\"ainen--Martio~\cite{HeKiMa}.
As $v$ and $\phi$ are continuous,
$v \ge \phi$ everywhere in $U$.
Since $x_0$ is elliptic regular, it satisfies the Wiener test in 
Wiener's criterion. 
Hence it follows from Gariepy--Ziemer~\cite[Theorem~2.2]{GZ} 
(or Mal\'y--Ziemer~\cite[Theorem~4.28]{MaZi})
that 
$\lim_{U \ni x \to0} v(x)=\phi(0)=0$.

We need to create barrier families both for upper and lower regularity.
We begin with upper regularity and let $c \in \Qp$ be arbitrary.
Let 
\[
    w_{c,j} = (c^m+jv+ajt^2)^{1/m},
\]
where $  a= c^{m-1}m/2\diam \Theta$.
Then 
\[
   \partial_t w_{c,j} = \frac{2ajt}{m}(c^m+jv+ajt^2)^{1/m-1}
   \ge - \frac{2aj \diam \Theta}{m}c^{1-m},
\]
while $\Delta w_{c,j}^m = j \Delta v= -j$.
Hence
\[
    \partial_t w_{c,j} - \Delta w_{c,j}^m
    \ge   j \biggl(1-\frac{2a \diam \Theta}{m}c^{1-m}\biggr)=0,
\]
and thus $w_{c,j}$ is superparabolic in 
$\Theta$, by Theorem~\ref{thm-lsc-supersoln}.
As $\lim_{U \ni x \to0} v(x)=0$ and $v \ge \phi$,
it follows that $\{w_{c,j}\}$ is an upper
barrier family and therefore by Theorem~\ref{thm:barrier-char-upper},
$\xi_0=(0,0)$ is upper regular.

Next, we turn to lower regularity for positive boundary data.
Let again $c \in \Qp$ be arbitrary.
This time we let
\[
    u_{c,j} = \max\biggl\{c^m-jv-bj^{1/m}t^2,\frac{1}{j}\biggr\}^{1/m}
      \quad \text{and} \quad
      \Theta_{c,j}=\biggl\{(x,t) \in \Theta : \ucj(x,t)^m > \frac{1}{j}\biggr\},
\]
where $b=m/2\diam \Theta$ and $v$ is the same function as above.
In $\Theta_{c,j}$, 
which is open as $v$ is continuous,
we have
\[
   \partial_t u_{c,j} = -\frac{2bj^{1/m}t}{m}(c^m-jv-bj^{1/m}t^2)^{1/m-1}
   \le  \frac{2bj^{1/m} \diam \Theta}{m}j^{1-1/m}
   =  j, 
\]
while $\Delta u_{c,j}^m = -j \Delta v= j$.
Hence
\[
    \partial_t u_{c,j} - \Delta u_{c,j}^m
    \le   j-j =0,
\]
and thus $u_{c,j}$ is subparabolic in 
$\Theta_{c,j}$, by Theorem~\ref{thm-lsc-supersoln}.
As $v$ is continuous, 
$u_{c,j}$ is also continuous.
Hence, Lemma~\ref{lem-pasting-const-sub} shows that $u_{c,j}$ is
subparabolic in $\Theta$.
Since $\lim_{U \ni x \to0} v(x)=0$ and $v \ge \phi$, 
it follows that $\{u_{c,j}\}$ is a lower
barrier family for positive boundary data
and thus by Theorem~\ref{thm:barrier-char-lower-pos},
$\xi_0=(0,0)$ is lower regular for positive boundary data.

Assume now instead that $\xi_0=(0,0)$ is regular for
positive boundary data. 
We let $\psi$ be a continuous function on $\bdy U$ and let 
$h$ be the harmonic Perron solution with boundary values $\psi$
with respect to $U$.
We need to show that $\lim_{U \ni x \to 0} h(x)=\psi(0)$.
As we can scale and add constants to harmonic functions
we may assume that $1 \le \psi \le 2$, and thus also $1 \le h \le 2$.
Let 
\[
    f(x,t)=\begin{cases}
      \psi(x), & \text{if } x \in \bdy U \text{ and  }t_1 < t < t_2, \\
      h(x), & \text{if } x \in \itoverline{U}, \text{ and }t=t_1$ or $t=t_2. 
      \end{cases}
\]
By Theorem~\ref{thm-indep-future} we may assume that $t_2 >0$.
Then $1 \le f \le 2$ on $\bdy \Theta$ and $f$ is continuous at $\xi_0$.
Moreover, 
if $u$ belongs to the elliptic lower  class for the
harmonic Perron solution of $\psi$, then
$\ut(x,t)=\max\{u(x),1\} \in \LL_f$ 
and hence $\lP f(x,t) \ge h(x)$ for $(x,t) \in \Theta$.
By Proposition~\ref{prop-reg-S-P-equiv}\,\ref{1-u}, 
\[
    \limsup_{U \ni x \to 0} h(x) 
    \le     \limsup_{U \ni x \to 0} \lP f(x,0) 
    \le     \limsup_{U \ni x \to 0} \uS f(x,0) 
    \le f(0,0)
    = \psi(0).
\]
It follows similarly that 
$    \liminf_{U \ni x \to 0} h(x) \ge \psi(0)$,
and thus 
$    \lim_{U \ni x \to 0} h(x) = \psi(0)$.
Hence  $0$ is elliptic regular with respect to $U$.
\end{proof}

A related result is proved by Ziemer in \cite[Theorem~4.4]{ziemer}. 
He considers general degenerate parabolic equations, which include 
the porous medium equation with $m>1$ as a special case. 
He deals with \emph{signed} weak solutions $u$ in a cylinder $U_{t_1,t_2}$, 
and assumes them to be bounded. 
The boundary data $f$ belong 
to the Sobolev space $W^{1,1}_2(\R^{n+1})$ 
of functions which, 
together with their distributional first derivatives, 
belong to $L^2(\R^{n+1})$, and $f{|}_{\R^{n+1}\backslash U_{t_1,t_2}}$ is continuous. 
The boundary condition on $\partial U_{t_1,t_2}$ is taken in a weak 
(Sobolev) sense, 
that is $u-f$ is assumed to be in the $W^{1,1}_2(U_{t_1,t_2})$-closure 
of smooth functions with compact support in $U_{t_1,t_2}$. 
If $\xi_0=(x_0,t_0)\in\partial U\times(t_1,t_2)$ and $x_0$ is elliptic regular, 
then by~\cite[Theorem~4.4]{ziemer},
\[
\lim_{U_{t_1,t_2}\ni\xi\to\xi_0} u(\xi)=f(\xi_0).
\]
There is no restriction on the boundary behaviour of $f$ at $\xi_0$, 
which can be positive, negative, or vanish, but the condition 
for continuity is only proved to be
sufficient.

Our aim in the rest of this section is to obtain
the following generalization of Abdulla's unique solvability result 
(Theorem~\ref{thm-cont-exist})
for $m\ge1$ and positive boundary data.
Note that it is a generalization also in the case when $\Theta$
is just one cylinder.

\begin{thm}   \label{thm-reg-monot-cyl}
Let $m\ge1$ and $\Theta$ be a finite union of cylinders.
Assume that the time-sections 
\[
\Theta(T)=\{(x,t)\in\Theta:t=T\}
\]
are bounded elliptic regular open sets in $\R^n$ satisfying
\[
\Theta(T_1)\subset\Theta(T_2) 
\quad \text{whenever }
T_{\min} < T_1<  T_2< T_{\max},
\]
where 
$T_{\min}:=\inf\{t:(x,t) \in \Theta\}$
and $T_{\max}:=\sup\{t:(x,t) \in \Theta$\}.
Then 
\begin{enumerate}
\item  \label{mon-cyl-reg}
every point in the parabolic boundary $\bdy_p\Theta$
is regular for positive boundary data\/\textup{;}
\item  \label{mon-cyl-res-pos}
every positive $f \in C(\bdy_p\Theta)$ is resolutive and 
\begin{equation} \label{eq-Pf-identities}
u:=Pf:=\lS f = \lP f = \uP f = \uS f 
\end{equation}
is the unique function in $C(\overline{\Theta})$ which is
parabolic in $\Theta$ and takes the boundary values $u=f$ on the
parabolic boundary $\bdy_p\Theta$\/\textup{;}
\item \label{mon-cyl-res}
every nonnegative $f \in C(\bdy_p\Theta)$ is resolutive
and for every $\xi_0\in\bdy_p\Theta$, either
\begin{equation} \label{eq-dichotomy}
\lim_{\Theta\ni\xi\to\xi_0} Pf(\xi) = f(\xi_0)
\quad \text{or} \quad
0 = \liminf_{\Theta\ni\xi\to\xi_0} Pf(\xi) 
     \le \limsup_{\Theta\ni\xi\to\xi_0} Pf(\xi) \le f(\xi_0).
\end{equation}
\end{enumerate}
\end{thm}

The last part also
generalizes the resolutivity result in
Kinnunen--Lindqvist--Lukkari~\cite[Theorem~5.1]{Kinnunen-Lindqvist-Lukkari:2013}
(for general cylinders) to certain finite unions of cylinders.

We will divide the proof into several results. 
Since they have independent interest we formulate
them in greater generality than Theorem~\ref{thm-reg-monot-cyl}.

\begin{thm}   \label{thm-unique-cont-union-cyl}
Let $m\ge1$ and $\Theta$ be a finite union of cylinders.
Let $h \in C(\bdy \Theta)$ be nonnegative.
Then there is at most one $u \in C(\overline{\Theta})$
that is parabolic in $\Theta$
and takes the boundary values $u=h$ on the parabolic boundary 
$\partial_p \Theta$. 
\end{thm}

\begin{proof}
Let $u$ and $v$ be two solutions of the boundary value problem 
under consideration.
Theorem~\ref{thm-para-comp-princ-union-cyl} shows that $v\le u\le v$ in
$\Theta$ and hence, by continuity, in~$\overline{\Theta}$.
\end{proof}

\begin{prop}  \label{prop-S=P}
Let $m >0$.
For  $T\in\R$, let
\[
\bdy_T\Theta:=\{\xi=(x,t)\in\bdy\Theta: t<T\}
\]
and assume that 
$f:\bdy\Theta_T\to[0,\infty)$ is bounded on $\bdy \Theta_T$
and continuous on $\bdy_T \Theta$.
Then the following are true.
\begin{enumerate}
\item  \label{T-ur}
If every $\xi\in\bdy_T\Theta$ is upper regular, then
\[
\uP_{\Theta_T} f \le \uS_{\Theta_T} f = \lP_{\Theta_T} f. 
\] 
\item \label{T-lrp}
If every $\xi\in\bdy_T\Theta$ is lower regular for positive boundary data 
and $f$ is bounded away from $0$, then
\[
\lP_{\Theta_T} f \ge \lS_{\Theta_T} f = \uP_{\Theta_T} f. 
\] 
\item \label{T-lr}
If every $\xi\in\bdy_T\Theta$ is lower regular 
for nonnegative boundary data, then
\[
\uP_{\Theta_T} f \le \lP_{\Theta_T} f. 
\] 
\item \label{T-res}
If the comparison principle $\lP_{\Theta_T} f \le \uP_{\Theta_T} f$
is valid on $\Theta_T$ then in any of the above cases \ref{T-ur}--\ref{T-lr}, 
$f$ is resolutive in $\Theta_T$.
\item \label{T-res-ae}
Let $0<h \in C(\bdy \Theta_T)$ be positive.
Then in any of the above cases \ref{T-ur}--\ref{T-lr}, 
there is a countable set $A \subset (0,1)$ such that 
\[
\lP_{\Theta_T} (f+ah)=  \uP_{\Theta_T} (f+ah)
  \quad \text{for every }a \in (0,1) \setm A,
\]
i.e.\ $f+ah$ is resolutive in $\Theta_T$.
\end{enumerate}
\end{prop}

This result naturally combines with Lemma~\ref{lem-restrict-Theta-T} 
and Corollary~\ref{cor-ThetaT}.
In particular, it can be applied to bounded continuous functions
defined only on $\bdy_T\Theta$ and extended 
arbitrarily in a bounded way to $\bdy\Theta_T$ and $\bdy\Theta$.

\begin{proof}
\ref{T-ur}
Proposition~\ref{prop-reg-S-P-equiv}\,\ref{1-u}, 
together with the upper regularity assumption, shows that 
\[
\limsup_{\Theta \ni \zeta \to \xi} \uS_{\Theta_T} f(\zeta) 
        \le f(\xi)
\quad \text{for every }\xi\in\bdy_T\Theta.
\]
Since $\uS_{\Theta_T}f$ is parabolic in $\Theta_T$, this
together with Lemma~\ref{lem-restrict-Theta-T}
(applied to $\Theta_T$)
shows that
\[
\uS_{\Theta_T}f \le \lP^T f = \lP_{\Theta_T}f.
\]
The converse inequality $\lP_{\Theta_T}f \le \uS_{\Theta_T}f$
follows 
from Theorem~\ref{thm-comp-princ}, 
and Remark~\ref{rmk-perron-domain} shows that 
$\uP_{\Theta_T} f \le \uS_{\Theta_T} f$.

\ref{T-lrp}
This is shown in the same way as \ref{T-ur}, by interchanging the role
of the upper and lower solutions, and taking into account the positivity
of $f$ and Proposition~\ref{prop-reg-S-P-equiv}\,\ref{1-r}.

\ref{T-lr}
As in \ref{T-ur}, Proposition~\ref{prop-reg-S-P-equiv}\,{\ref{1-nn}}, 
together with the lower regularity assumption, implies that 
\[
\limsup_{\Theta \ni \eta \to \xi} \lP_{\Theta_T} f(\eta) 
        \ge f(\xi)
\quad \text{for every }\xi\in\bdy_T\Theta,
\]
and hence, by Lemma~\ref{lem-restrict-Theta-T} (applied to $\Theta_T$),
$\lP_{\Theta_T}f \ge \uP_{}^T f = \uP_{\Theta_T}f$.

\ref{T-res}
This is a direct consequence of \ref{T-ur}--\ref{T-lr} and the 
general inequalities
\[
\lS_{\Theta_T}f \le \lP_{\Theta_T}f \le \uS_{\Theta_T}f
\quad \text{and} \quad
\lS_{\Theta_T}f \le \uP_{\Theta_T}f \le \uS_{\Theta_T}f.
\]

\ref{T-res-ae}
Let $E$ be a countable dense subset of $\Theta_T$.
For each $\xi \in E$ the function
$(0,1) \ni a \mapsto \lP(f+ah)(\xi)$ is nondecreasing and
thus has jumps for at most countably many values of $a$;
let $A_{\xi}$ be the set of these values of $a$.
Then $A=\bigcup_{\xi \in E} A_{\xi}$ is also countable.

Now let $a \in  (0,1) \setm A$. 
Then, for every $\xi \in E$, using the positivity of $h$ 
on $\bdy\Theta_T$ and the elliptic comparison principle
(Theorem~\ref{thm-comp-princ}), together with
\ref{T-ur}, \ref{T-lrp} or \ref{T-lr} above,
we see that 
\[
    \lP (f+a h)(\xi) = \lim_{b \to a\limminus} \lP (f+bh)(\xi)
    \le \uP (f+a h)(\xi)
    \le \lP (f+a h)(\xi).
\]
Thus $\lP (f+ah)= \uP (f+ah)$ in $E$ and as they are both continuous
this holds everywhere in $\Theta_T$, i.e.\ $f+ah$ is resolutive.
\qedhere
\end{proof}

\begin{thm}   \label{thm-dichotomy}
Let $m\ge1$ and let $f:\bdy\Theta\to [0,\infty)$ be bounded. 
Assume in addition that $f$ is continuous at $\xi_0\in\bdy\Theta$, $f(\xi_0)>0$
and that $\xi_0$ is lower regular for positive boundary data.
Then either 
\[
\liminf_{\Theta\ni\xi\to\xi_0} \uP f(\xi) \ge f(\xi_0)
\quad \text{or} \quad
\liminf_{\Theta\ni\xi\to\xi_0} \uP f(\xi) = 0.
\]
Similarly, either
$\liminf_{\Theta\ni\xi\to\xi_0} \uS f(\xi) \ge f(\xi_0)$
or $\liminf_{\Theta\ni\xi\to\xi_0} \uS f(\xi) = 0$.
\end{thm}

\begin{proof}
Assume that $\liminf_{\Theta\ni\xi\to\xi_0} \uP f(\xi) >0$.
Then there exists $\eps>0$ and a ball $B\ni\xi_0$, such that 
$\uP f\ge \eps$ on $\Theta\cap\itoverline{B}$ and 
$f\ge \eps$ on $\bdy\Theta\cap \itoverline{B}$.
Thus, the function
\[
h=\begin{cases} f & \text{on } \bdy\Theta\cap \itoverline{B}, \\
                \uP f & \text{on } \bdy B \cap\Theta,
   \end{cases}
\]
is bounded away from zero on $\bdy(B\cap\Theta)$ and continuous at $\xi_0$.

Now, let $u\in\UU_f(\Theta)$ be arbitrary.
Then by definition,
\[
\liminf_{\Theta\ni{\zeta\to\xi}} u(\zeta)\ge f(\xi) =h(\xi)
\quad \text{for  } \xi\in\bdy\Theta\cap \itoverline{B},
\]
while the lower semicontinuity of $u$, together with the definition of $\uP f$
and $h$, implies that
\[
\liminf_{\Theta\ni{\zeta\to\xi}} u({\zeta}) \ge u({\xi}) 
\ge\uP f({\xi}) =h({\xi})
\quad \text{for  } {\xi}\in\bdy B\cap\Theta.
\]
It follows that $u\in \UU_h(\Theta\cap B)$ and taking infimum over all
such $u$ yields 
\[
\uP f \ge \uP_{B\cap\Theta}h \ge \lS_{B\cap\Theta}h 
\quad \text{in } B\cap\Theta.
\]
Corollary~\ref{cor-restrict-pos} shows that $\xi_0$ is lower regular 
for positive boundary data with respect to $B\cap\Theta$.
We therefore conclude from Proposition~\ref{prop-reg-S-P-equiv}\,\ref{1-r}
that
\[
\liminf_{\Theta\ni\zeta\to\xi_0} \uP f(\zeta) 
\ge \liminf_{\Theta\ni\zeta\to\xi_0} \lS_{B\cap\Theta}h(\zeta) 
\ge h(\xi_0) = f(\xi_0).
\]
The case with $\uS f$ is obtained similarly.
\end{proof}

Finally, we are ready to prove Theorem~\ref{thm-reg-monot-cyl}.

\begin{proof}[Proof of Theorem~\ref{thm-reg-monot-cyl}]
\ref{mon-cyl-reg}
There are two types of points in $\bdy_p\Theta$: those belonging
to the lateral surface of one of the cylinders constituting $\Theta$ and
those belonging to the flat bottom of one of these cylinders.
Because of the assumption $\Theta(T_1)\subset\Theta(T_2)$, none of the
parabolic boundary points belongs to the top of any of the cylinders.
(Here by the \emph{top} of $U_{t_1,t_2}$ we mean 
$\bdy U_{t_1,t_2} \setm \bdyp U_{t_1,t_2}$.)

The lateral points are regular for positive boundary data 
by Theorem~\ref{thm-reg-cylinder},
together with Propositions~\ref{prop-local-up} and~\ref{prop-local-pos},
while the bottom points are regular for positive boundary data by 
Proposition~\ref{prop-earliest-reg-below-new}.

\ref{mon-cyl-res-pos}
The resolutivity and the identities in \eqref{eq-Pf-identities}
follow from Proposition~\ref{prop-S=P}\,\ref{T-ur}, \ref{T-lrp}
and~\ref{T-res}, together with the parabolic
comparison principle, Theorem~\ref{thm-para-comp-princ-union-cyl}.
The uniqueness is a direct consequence of 
Theorem~\ref{thm-unique-cont-union-cyl}.

The continuity of $u$ on $\bdy_p\Theta$ follows from 
Proposition~\ref{prop-reg-S-P-equiv}\,\ref{1-ur}, while Theorem~5.16.1
in DiBenedetto--Gianazza--Vespri~\cite{DBGV-mono} shows that $u$
has a continuous extension to all points on the top level
$\bdy\Theta\setm\bdy_p\Theta$.

\ref{mon-cyl-res}
As in \ref{mon-cyl-res-pos}, the resolutivity follows from 
Proposition~\ref{prop-S=P}\,\ref{T-ur}
and~\ref{T-res}, together with Theorem~\ref{thm-para-comp-princ-union-cyl}.
Proposition~\ref{prop-reg-S-P-equiv}\,\ref{1-u} then shows that 
\[
\limsup_{\Theta\ni\xi\to\xi_0} Pf(\xi) \le f(\xi_0)
\quad \text{for all }\xi_0\in\bdy_p\Theta,
\]
while Theorem~\ref{thm-dichotomy} implies that \eqref{eq-dichotomy} holds.
\end{proof}

\appendix
%
%
\makeatletter
\def\@seccntformat#1{Appendix \csname the#1\endcsname.\quad}
\makeatother

\section{Proof of Theorem~\ref{thm-usc-repr}}
\label{app-pf-usc}

Kuusi~\cite{Kuusi09}, working with the \p-parabolic 
equation~\eqref{eq-p-para},  observed that 
the crucial step towards establishing inner regularity
results like Theorem~\ref{thm-usc-repr} 
is the following supremum estimate  for subsolutions.
Such results had earlier been obtained using weak Harnack inequalities,
which are more difficult to deduce than supremum estimates.
Avelin--Lukkari~\cite{AvelinL2015} later adapted Kuusi's argument
to supersolutions of the porous medium equation (establishing
Theorem~\ref{thm-lsc-repr}).

First, we set the  notation
\[
   Q(x_0,t_0,\rho)=B(x_0,\rho) \times (t_0-\rho^2,t_0+\rho^2)
 \quad \text{and} \quad
   Q(\rho)=Q(0,0,\rho).
\]

\begin{prop} \label{prop-app}
\textup{(Supremum estimate for subsolutions)}
Let $m \ge 1$, $u$ be a bounded nonnegative subsolution in $Q(x_0,t_0,\rho)$,
$0 < \sigma <1$ and $M \ge 0$.
Then 
\begin{equation}   \label{bound-subsol-final}
\esssup_{Q(x_0,t_0,\sigma \rho)} (u-M) \\
\le C \biggl(\frac1{|Q(\rho)|}
\iint_{Q(x_0,t_0,\rho)}(u-M)_\limplus^2\,dx\,dt\biggr)^{1/\lambda},
\end{equation}
where $\lambda=2+\frac{(m-1)n}2$, and $C$ only depends on $n$, $m$, 
$\sigma$ and $L:=\esssup_{Q(x_0,t_0,\rho)}u$, but not on $u$, $M$ and $\rho$.
\end{prop}

To prove this estimate we modify the technique
used by Andreucci~\cite{andreucci}.
Note that we do not have the extra corrective term that appears in 
\cite{andreucci}, 
since here we directly assume that the height of the cylinder is $\rho^2$ 
and allow $C$ to depend on $L$, which is sufficient for our purposes.

Before proving Proposition~\ref{prop-app}, 
we show how this estimate is used to obtain Theorem~\ref{thm-usc-repr}.

\begin{proof}[Proof of Theorem~\ref{thm-usc-repr}]
Let $(x_0,t_0) \in \Theta$ be a Lebesgue point of $u$.
Without loss of generality we assume that $(x_0,t_0)=(0,0)$. 
We want to show that $u^*(0,0)=u(0,0)$. 
First of all,
\[
u^*(0,0)\ge\lim_{\rho\to 0}\frac1{|Q(\rho)|}\iint_{Q(\rho)}u\,dx\,dt=u(0,0).
\]

For the converse inequality, first 
choose $r$ such that $Q(r)\Subset \Theta$.
Let $L=\esssup_{Q(r)} u$, which is finite, see e.g.\ Andreucci~\cite{andreucci}.
Then,  by Proposition~\ref{prop-app}
(with $M=u(0,0)$ and $\sigma=\tfrac12$), we have for any $0<\rho<r$,
\begin{align*}
\esssup_{Q(\rho/2)}(u-u(0,0))
&\le C \biggl(\frac1{|Q(\rho)|}
  \iint_{Q(\rho)}(u-u(0,0))_\limplus^2\,dx\,dt\biggr)^{1/\lambda}\\
&\le C L^{1/\lambda}\biggl(\frac1{|Q(\rho)|}
\iint_{Q(\rho)}|u-u(0,0)|\,dx\,dt\biggr)^{1/\lambda},
\end{align*}
which tends to $0$, as $\rho \to 0$, since $(0,0)$ is a Lebesgue point.
Thus we conclude that $u^*(0,0) \le u(0,0)$.
\end{proof}

Now we turn to the proof of Proposition~\ref{prop-app}.

\begin{proof}[Proof of Proposition~\ref{prop-app}]
Without loss of generality we assume that $(x_0,t_0)=(0,0)$.
Define for all $j\ge0$,
\begin{alignat*}{2}
t_j^\limpm& =\pm \biggl(\sigma^2\rho^2+\frac{1-\sigma^2}{2^{j}}\rho^2\biggr),
& \quad 
\rho_j& =\sigma\rho+\frac{1-\sigma}{2^j}\rho, \\
B_j& =B(0,\rho_j), 
& \quad 
Q_j &= B_j 
\times(t_j^\limminus,t_j^\limplus).
\end{alignat*}
Note that
\[
Q(\sigma \rho) \Subset Q_{j+1} \Subset Q_j \Subset Q_0=Q(\rho).
\]
We also consider $C^\infty$-cutoff functions
$\zeta_j$ such that $0 \le \zeta_j \le 1$ and
\begin{alignat}{2}  \label{eq-est-for-zeta}
\zeta_j& \equiv0 \text{ outside }  Q_j,
& \quad 
\zeta_j& \equiv1 \text{ in } Q_{j+1}, \nonumber
\\
|\nabla\zeta_j|& \le\frac{2^{j+2}}{(1-\sigma)\rho},
& \quad   
0  \le |\partial_t\zeta_j| & \le\frac{2^{j+2}}{(1-\sigma^2)\rho^2}.
\end{alignat}
Fix $k>0$ to be chosen later, set 
\[
\uj = (u-M-k_{j+1})_\limplus, \quad 
\text{where } k_j=k-\frac{k}{2^{j+1}},
\]
and use
$ f_j= 2 \uj \zeta_j^2 $
as a test function.
Note that 
\begin{equation}   \label{eq-grad-uj=u}
\bdy_t \uj = \bdy_t u \quad \text{and} \quad \grad \uj = \grad u
\end{equation} 
a.e.\ in the set where $\uj\ne 0$
(and  that $\bdy_t\uj = \grad\uj = 0$ a.e.\ otherwise).
Some of the calculations below are formal.
As far as the time derivative $\partial_t u_j$ 
is concerned, the calculations
can be made rigorous by means of 
a Steklov averaging process, cf.\ e.g.\ 
DiBenedetto--Gianazza--Vespri~\cite[pp.~21 and~35]{DBGV-mono}, or by a
mollification in time, see
e.g.\ Kinnunen--Lindqvist~\cite[pp.~141--143]{Kinnunen-Lindqvist:2008}. Another
difficulty is represented by $\nabla u$, since in general only $\nabla
u^m$ is well-defined: for a way to deal with this second issue, see
e.g.\ B\"ogelein--Duzaar--Gianazza~\cite[Lemma~2.2]{BDG1}.
Note, however, that in the integrals below, $\grad u$ is only considered
at points, where $u>k_j$, and thus $\grad u = m^{-1}(u^m)^{1/m-1}\grad u^m$
is well defined.

For all $\tb\in(t_j^\limminus,t_j^\limplus)$, the time part of 
the weak formulation of 
subsolutions, within $\Qjt:=B_j\times(t_j^\limminus,\tau)$,
becomes
\begin{align}\label{subsol-time-der}
\iint_{\Qjt} 
f_j \partial_t u \,dx\,dt 
&= \iint_{\Qjt} 
\zeta_j^2 \bdy_t \uj^2 \,dx\,dt
\\
&\ge \int_{B_j}\uj(x,\tb)^2 \zeta_j(x,\tb)^2 \,dx
-\frac{2^{j+3}}{(1-\sigma^2)\rho^2} \iint_{\Qjt} 
\uj^2\,dx\,dt \nonumber,
\end{align}
where we used partial integration with respect to $dt$
and properties~\eqref{eq-est-for-zeta} of the cutoff function $\zeta_j$ 
in the second step.

Now we turn to the elliptic part. 
We have
\[
\grad f_j = 2 \zeta_j^2 \grad \uj 
+ 4 \uj \zeta_j \grad \zeta_j \quad
\text{a.e.}
\]
and hence, using~\eqref{eq-grad-uj=u}, we see that
\begin{align*}
 \iint_{\Qjt}\nabla u^m\cdot\nabla f_j \,dx\,dt  
&\ge 2m \iint_{\Qjt} u^{m-1} \zeta_j^2 |\grad \uj|^2 \,dx\,dt\\
&\quad 
  -4m \iint_{\Qjt} u^{m-1} ( \zeta_j |\grad \uj|) (\uj 
             |\grad\zeta_j|) \,dx\,dt. 
\end{align*}
Next, using 
Young's inequality, we get that for every $\theta>0$,
\begin{align*}
\iint_{\Qjt}\nabla u^m \cdot \nabla f_j \,dx\,dt
&\ge (2m-2m\theta) \iint_{\Qjt} u^{m-1} \zeta_j^2  |\grad \uj|^2 \,dx\,dt\\
&\quad 
-\frac{2m}{\theta}\iint_{\Qjt} u^{m-1} \uj^2 |\grad\zeta_j|^2 \,dx\,dt.
\end{align*}
We now make the choice $\theta=\tfrac12$.
Using also that $L \ge u \ge k_{j+1} \ge k/2$ when $\uj\ne0$
(and  that $\grad \uj = 0$ a.e.\ otherwise), 
together with~\eqref{eq-est-for-zeta},
we conclude that
\begin{align*}
\iint_{\Qjt}\nabla u^m \cdot \nabla f_j \,dx\,dt
& \ge \frac{m k^{m-1}}{2^{m-1}} 
   \iint_{\Qjt} \zeta_j^2 |\grad \uj|^2 \,dx\,dt\\
&\quad 
   - \frac{4^{j+3} m L^{m-1}}{(1-\sigma)^2\rho^2}
             \iint_{\Qjt} \uj^2 \,dx\,dt.
\end{align*}
Combining this with \eqref{subsol-time-der} and 
using that $u$ is a subsolution yields the energy estimate
\[ 
      \int_{B_j} \uj(x,\tau)^2 \zeta_j(x,\tau)^2 \,dx 
+ k^{m-1} 
   \iint_{\Qjt} \zeta_j^2 |\grad \uj|^2 \,dx\,dt 
\le \frac{4^{j}\gah}{2\rho^2}  \iint_{\Qjt} \uj^2\,dx\,dt,
\]
where $\gah$ only depends on $m$, $\sigma$ and $L$, 
but not on $u$, $M$, $k$, $\rho$ or $j$.
Taking supremum separately of each term on the left-hand side
and adding the estimates, yields
\begin{align}
\sup_{t_j^\limminus\le t\le t_j^\limplus}
      \int_{B_j} \uj(x,t)^2 \zeta_j^2(x,t) \,dx
&+ k^{m-1} 
   \iint_{Q_j} \zeta_j^2 |\grad \uj|^2 \,dx\,dt \nonumber\\ 
& 
\le \frac{4^{j}\gah}{\rho^2}  \iint_{Q_j} \uj^2\,dx\,dt
\le \frac{4^{j}\gah}{\rho^2} Y_j,
\label{energy-est}
\end{align}
where
\[
Y_j := \iint_{Q_j} (u(x,t)-M-k_{j})^2_\limplus\,dx\,dt.
\]
Next, we need an estimate for the measure of 
\[
 A_j =\{(x,t) \in Q_j : u_j(x,t)>0\}
 =\{(x,t)\in Q_j: u(x,t) > M + k_{j+1} \}.
\]
Using that $k_{j+1}-k_j = 2^{-(j+2)}k$, we see that
\begin{equation} \label{est-Ajp1}
Y_j 
\ge \iint_{A_{j}} (u(x,t)-M-k_{j})^2_\limplus\,dx\,dt 
\ge  4^{-(j+2)} k^2  |A_j|.
\end{equation}
Let $\alp=2/(n+2)$ and note that $1-\alp=n/(n+2)$.
By H\"older's inequality and 
the parabolic 
Sobolev embedding 
(see e.g.\ DiBenedetto--Gianazza--Vespri~\cite[Proposition~4.1,
Chapter~1]{DBGV-mono}),
we obtain that
\begin{align*}
Y_{j+1} &\le \iint_{Q_j} (\uj \zeta_j)^2\,dx\,dt \\
&\le |A_{j}|^{\alp} \biggl( \iint_{Q_j} (\uj \zeta_j)^{2/(1-\alp)} 
              \,dx\,dt \biggr)^{1-\alp} \\
& \le \gamma(n) |A_{j}|^{\alp} \biggl( \iint_{Q_j}
     (|\zeta_j \grad\uj|^2 + |\uj\grad\zeta_j|^2) 
              \,dx\,dt \biggr)^{1-\alp} \\
&\quad  \times \biggl(\sup_{t_j^\limminus\le t\le t_j^\limplus} 
   \int_{B_j} \uj(x,t)^2 \zeta_j(x,t)^2 \,dx\,dt\biggr)^{\alp}.
\end{align*}
Estimating the factors on the right-hand side using 
\eqref{est-Ajp1} and \eqref{energy-est}
we can conclude that
\[
Y_{j+1}\le A b^j\, Y_j^{1+\alp},
\]
where
\[
b:=4^{1+\alp} \quad \text{and} \quad
A:= \frac{\gat (1+ k^{1-m})^{1-\alp}}{k^{2\alp}\rho^2},
\]
and $\gat$ only depends on $n$, $m$, $\sigma$ and $L$,
but not on $u$, $M$, $k$, $\rho$ or $j$.
It is now easily shown by induction 
(cf.\ DiBenedetto--Gianazza--Vespri~\cite[Lemma~5.1, Chapter~1]{DBGV-mono}) 
that 
\[
Y_j \le A^{-1/\alp} b^{-1/\alp^2} b^{-j/\alp} \to 0  \text{ as } j\to\infty, 
\quad \text{i.e. } u\le M+k  \text{ in } Q(\sigma\rho),
\]
provided that $k$ is chosen so that 
$Y_0 \le A^{-1/\alp} b^{-1/\alp^2}$.
Such a condition is satisfied if 
\[
Y_0 \le \iint_{Q_0} (u(x,t)-M)_\limplus^2\,dx\,dt
\le \gat^{-1/\alp}  b^{-1/\alp^2} \frac{k^2\rho^{n+2}}
   {(1+ k^{ 1-m})^{n/2}}.
\]
Choosing
\[
k = \ga' \biggl( \biggl( \frac{Y_0}{\rho^{n+2}} \biggr)^{1/\la}
        + \biggl( \frac{Y_0}{\rho^{n+2}} \biggr)^{1/2} \biggr),
\]
with $\ga'$ large enough only depending on $n$, $m$, $\sigma$ and $L$, will do.
Since $Y_0/\rho^{n+2} \le 2|B(0,1)|L^2$ and $1/2-1/\la>0$, 
we see that
\[
       \esssup_{Q(\sigma\rho)} {(u-M)} \le
    k \le \ga \biggl( \frac{Y_0}{\rho^{n+2}} \biggr)^{1/\la},
\]
where $\ga$ again only depends on $n$, $m$, $\sigma$ and $L$.
From this
\eqref{bound-subsol-final} follows.
\end{proof}

\section*{Conflict of Interest}

\noindent
{\bf Funding:} 
A.\ B. and J.\ B.\ were supported by the Swedish Research Council, 
Grant Nos.\ 621-2011-3139, 621-2014-3974 and 2016-03424. 
J.\ S. was supported by the Academy of Finland grant 259363
and a V\"ais\"al\"a foundation travel grant.

\medskip

\noindent
{\bf Conflict of Interest:}
 The authors declare that they have no conflict of interest.

\end{document}